\documentclass[12pt,a4paper,oneside]{amsart}
\usepackage[utf8]{inputenc}
\usepackage{amsmath,amscd, amsfonts, amssymb, amsthm,}
\usepackage[colorlinks=true, citecolor=blue, linkcolor=blue]{hyperref}
\usepackage{setspace,kantlipsum}
\usepackage{tikz-cd}
\usepackage{mathrsfs}
\usepackage{textcomp}
\usepackage[margin=1in]{geometry}
\usepackage{colonequals}
\usepackage{mathtools}
\usepackage{stackrel}
\usepackage{amssymb}

\theoremstyle{plain} 
\newtheorem{thm}{Theorem}[section]
\newtheorem{lemma}[thm]{Lemma}
\newtheorem{prop}[thm]{Proposition}
\newtheorem{corollary}[thm]{Corollary}

\theoremstyle{definition}

\newtheorem{definition}[thm]{Definition}

\newtheorem{example}[thm]{Example}
\newtheorem{remark}[thm]{Remark}

 \theoremstyle{plain} 
\newcommand{\thistheoremname}{}
\newtheorem{genericthm}[thm]{\thistheoremname}

\newtheorem*{genericthm*}{\thistheoremname}
\newenvironment{namedthm*}[1]
  {\renewcommand{\thistheoremname}{#1}%
   \begin{genericthm*}}
  {\end{genericthm*}}

\newcommand{\C}{\mathbb{C}}

\newcommand{\Q}{\mathbb{Q}}

\newcommand{\Z}{\mathbb{Z}}
\newcommand{\R}{\mathbb{R}}

\newcommand{\cF}{\mathcal{F}}

\newcommand{\cH}{\mathcal{H}}
\newcommand{\cI}{\mathcal{I}}
\newcommand{\cJ}{\mathcal{J}}

\newcommand{\cM}{\mathcal{M}}
\newcommand{\cN}{\mathcal{N}}
\newcommand{\cO}{\mathcal{O}}

\newcommand{\cV}{\mathcal{V}}

\newcommand{\sD}{\mathscr{D}}

\newcommand{\sO}{\mathscr{O}}

\newcommand{\mc}[1]{{\mathcal{#1}}}
\newcommand{\mb}[1]{{\mathbb{#1}}}

\renewcommand{\d}{\partial}

\newcommand{\gr}{\mathrm{gr}}

\newcommand{\Db}{\mathrm{Db}}
\newcommand{\shom}{{\mathcal{H}\mathit{om}}}

\newcommand{\lct}{\mathrm{lct}}
\renewcommand{\Re}{\mathop{\mathrm{Re}}}
\DeclareMathOperator*{\mult}{mult}
\newcommand{\op}{\operatorname}
\DeclareMathOperator{\codim}{codim}
\DeclareMathOperator*{\Res}{Res}
\newcommand{\contract}{\mathrel{\lrcorner}}
\DeclareMathOperator{\supp}{supp}

\title[Archimedean zeta function and Hodge theory]{Archimedean zeta functions, singularities, and Hodge theory}
\author{Dougal Davis, Andr\'as C. L\H{o}rincz and Ruijie Yang}
\thanks{DD was supported by the ARC grants FL200100141 and DP250100824.}
\begin{document}

\begin{abstract}
We use Hodge theory to relate poles of the Archimedean zeta function $Z_f$ of a holomorphic function $f$ with several invariants of singularities. First, we prove that the largest nontrivial pole of $Z_f$ is the negative of the minimal exponent of $f$, whose order is determined by the multiplicity of the corresponding root of the Bernstein--Sato polynomial $b_f(s)$, resolving in a strong sense a question of Musta\c{t}\u{a}--Popa. On the other hand, we give an example of $f$ where a root of $b_f(s)$ is not a pole of $Z_f$, answering a question of Loeser from 1985 in the negative. As a byproduct, we give a positive answer to a question of Budur--Walther in the case of the minimal exponent.  In general, we determine poles of $Z_f$ from the Hodge filtration on vanishing cycles, sharpening a result of Barlet.  Finally, we obtain analytic descriptions of the $V$-filtration of Kashiwara and Malgrange, Hodge and higher multiplier ideals, addressing another question of Musta\c{t}\u{a}--Popa. The proofs mainly rely on a positivity property of the polarization on the lowest piece of the Hodge filtration on a complex Hodge module in the sense of Sabbah--Schnell. 
\end{abstract}
\maketitle

\section{Introduction}

Let $f$ be a nonconstant holomorphic function on an open subset $X\subseteq\C^n$. For a smooth function $\varphi$ on $\C^n$ with compact support contained in $X$, the following integral converges to a holomorphic function of $s$:
\begin{equation}\label{eq:intro zeta function}
Z_f(\varphi; s) = \int_{\mb{C}^n} |f(x)|^{2s}\varphi(x)d\mu(x),  \quad \textrm{for $\Re s > 0$},\end{equation}
and extends meromorphically to $s \in \mb{C}$ (here $d\mu$ is the Lebesgue measure). The \emph{Archimedean zeta function} $Z_f$ is the distribution sending $\varphi$ to $Z_f(\varphi;s)$. 

In this paper, we use the theory of polarized complex Hodge modules \cite{MHMproject} to study the relationship between the poles of $Z_f$ and singularity invariants of the hypersurface $D$ defined by $f$. We show (Theorem \ref{thm: Loesers conjecture}) that the negative of the minimal exponent \cite{Saito93} is a pole, of a specified order, but that this need not hold for other roots of the Bernstein--Sato polynomial (Theorem \ref{thm: counterexample to Loeser}). Using similar methods, we find other poles (Theorem \ref{thm: producing poles}) via the Hodge filtration on vanishing cycles, and prove (Theorems \ref{thm:intro Hodge ideals}, \ref{thm:intro weighted ideals} and \ref{thm: analytic charaterization of Vfiltration}) that the (weighted) Hodge ideals \cite{MP16, OlanoweightedHodgeideal}, higher multiplier ideals \cite{SY23} and Kashiwara--Malgrange $V$-filtration can be read off from the poles of zeta functions. Our results answer a question of Loeser \cite{Loeser}, several questions of Musta\c{t}\u{a}--Popa \cite{Mustatazagtalk,Mustataprivate} and refine previous work of Barlet \cite{Barlet84,Barlet86}, Loeser \cite{Loeser}, and Koll\'ar--Litchin \cite{kollar,lichtin}. We also answer positively a question of Budur--Walther in the case of the minimal exponent (Corollary \ref{cor: Budur Walther question for minimal exponent}).

\subsection{The Bernstein--Sato polynomial}
Our first two results concern poles of $Z_f$ and roots of the \emph{Bernstein--Sato polynomial} $b_f(s)$ of $f$. Up to shrinking  $X$, $b_f(s)$ exists by \cite[Theorem 3.3]{Kas76}\footnote{Note that in the algebraic case this implies global existence, since algebraic varieties are quasi-compact in the Zariski topology.}. A \emph{pole} of $Z_f$ is a pole of $Z_f(\varphi;s)$ for some $\varphi$ whose support is contained in $X$. Bernstein \cite{bernstein} used $b_f(s)$ to construct $Z_f$, which implies that every pole of $Z_f$ is a root of $\prod_{i=0}^{r}b_f(s+i)$ with $r\gg 0$, and that the order as a pole can be at most the multiplicity as a root.  Our first main theorem gives a converse in the case of a special root, the negative of the minimal exponent, which has received significant attention in recent years \cite{MPVfiltration,SY23,RR24}. 

The \emph{minimal exponent} $\tilde{\alpha}_f$ is the negative of the largest root of $\tilde{b}_f(s) = b_f(s)/(s + 1)$. Similarly, $Z_f$ has trivial poles at the negative integers, which can be removed by defining the \emph{reduced zeta function}: $\tilde{Z}_f(\varphi;s) = \frac{Z_f(\varphi;s)}{\Gamma(s + 1)}$, where $\Gamma(s)$ denotes the gamma function. 
\begin{thm} \label{thm: Loesers conjecture}
The root $ -\tilde{\alpha}_f$ is a pole of $\tilde{Z}_f$, of order equal to its multiplicity as a root of $\tilde{b}_f(s)$. In particular, $-\tilde{\alpha}_f$ is the largest pole of $\tilde{Z}_f$.
\end{thm}
\noindent If $f$ has only isolated singularities, it was shown by Loeser \cite[Th\'eor\`eme 1.8]{Loeser} that $-\tilde{\alpha}_f$ is a pole of $\tilde{Z}_f$; but the pole order part is new 
in this setting. By Koll\'ar \cite{kollar} and Lichtin \cite{lichtin}, the \emph{log canonical threshold} $\lct(f)$ is $\min\{\tilde{\alpha}_f,1\}$, whose negative is the largest pole of $Z_f$ \cite[Theorem 10.6]{kollar}. In particular, Theorem \ref{thm: Loesers conjecture} recovers a result of Adiceam--Marmon \cite[Theorem 1.3]{AdiceamMarmon} that the order of $-\lct(f)$ as a pole of $Z_f$ is equal to $\textrm{mult}_{s=-\mathrm{lct}(f)}b_f(s)$. We note that in the lct case, this can be deduced from positivity of the convergent integral $Z_f(\varphi;s)$ near $s = -\lct(f)$ for certain $\varphi$, while for $\tilde{\alpha}_f$ we must appeal to Hodge theory to prove the analogous positivity properties (see below).

Theorem \ref{thm: Loesers conjecture} answers a question of Musta\c{t}\u{a}--Popa \cite[Page 20]{Mustatazagtalk} in the strong sense, and gives an analytic characterization of the minimal exponent $\tilde{\alpha}_f$ via $\tilde{Z}_f$. The proof also shows that the maximum pole order is detected by $\tilde{Z}_f(\varphi; s)$ for any non-negative test function $\varphi$ whose support intersects
\begin{equation}\label{eqn: locus of maximal pole order} \{ x\in X \mid \mult_{s=-\tilde{\alpha}_f} b_{f,x}(s)=\mult_{s=-\tilde{\alpha}_f} b_{f}(s)\},\end{equation}
where $b_{f,x}(s)$ is the local Bernstein--Sato polynomial of $f$ at $x$; \eqref{eqn: locus of maximal pole order} agrees with the \emph{center of minimal exponent} defined in \cite{SY23}, by Corollary \ref{corollary: determine the center of minimal exponent}. If $f$ is a polynomial, then $Z_f$ extends to a tempered distribution \cite[Theorem 5.3.2]{igusa}, and the maximal pole order  is detected by any everywhere positive rapidly decreasing function, e.g. $\varphi(x)=\exp(-|x|^2)$.

In some cases, Theorem \ref{thm: Loesers conjecture} can be used effectively to determine $\tilde{\alpha}_f$ without computing  $b_f(s)$, see Example \ref{example: Whitney umbrella} for an illustration. We remark that analogous questions for functions on $\R^n$, which was the original setting of Gelfand's problem \cite{gelfand}, can fail already for the log canonical threshold, see \cite{saitoreal}.

\smallskip

The question of which roots of the $b$-function are poles of $Z_f$ was initially raised by Barlet in \cite{Barlet82}. In his works, he considered the contribution of the monodromy eigenvalues of $f$ to the poles of $Z_f$. In contrast, Loeser \cite{Loeser} studied this without referring to monodromy, employing a direct method; our approach is closer to the one of Loeser. For a general root $s_0$ of $b_f(s)$, it was shown by Barlet \cite{Barlet84} that there exists $m\in \Z_{\geq 0}$ such that $s_0-m$ is a pole of $Z_f$. An effective control of $m$ is given in Theorem \ref{thm: producing poles}. It was asked by Loeser \cite[p87]{Loeser} whether $s_0$ itself is always a pole of $Z_f$. Our next result answers this in the negative. 
\begin{thm}\label{thm: counterexample to Loeser}
Let $\alpha\in (0,1)$. If $-\alpha$ is a simple root of $b_f(s)$ and $f^{-\alpha}\in \Gamma(X,\sD_{X} \cdot f^{-\alpha+1})$, then $-\alpha$ is not a pole of $Z_f$. 
\end{thm}

\noindent In \cite[Example 4.2]{saitopower}, Saito gives an example of $f$ on $\C^3$ with an isolated singularity and $\alpha = 13/30$ so that the assumption above is satisfied, in order to give a negative answer to a question of Budur and Walther on whether $b_f(-\alpha)=0$ implies $\sD_{X}\cdot f^{-\alpha}\neq \sD_{X}\cdot f^{-\alpha+1}$. Despite many explicit calculations in the literature (e.g., \cite[Th\'eor\`eme 1.9]{Loeser}, \cite[Theorem 6.3.1]{igusa} and \cite{blanco}), to our knowledge Theorem \ref{thm: counterexample to Loeser} gives the \emph{first} example of a root that is not a pole and shows that the shifts are necessary.

As a byproduct, using Theorem \ref{thm: Loesers conjecture} and an enhancement of Theorem \ref{thm: counterexample to Loeser} (see Proposition \ref{prop: an estimate for pole order}), we show that the question of Budur--Walther has a positive answer for $\tilde{\alpha}_f$.
\begin{corollary}\label{cor: Budur Walther question for minimal exponent}
    We have $\sD_{X}\cdot f^{-\tilde{\alpha}_f}\neq \sD_{X} \cdot f^{-\tilde{\alpha}_f+1}$.
\end{corollary}

\subsection{Complex polarized Hodge modules}
For a holomorphic function $f$ on an arbitrary complex manifold $X$, one can define $b_f(s)$ and $Z_f$ (resp. $\tilde{Z}_f$) exactly as for $\mb{C}^n$ (see (\ref{eq: original zeta})),
and Theorem \ref{thm: Loesers conjecture} extends verbatim. We work in this generality from now on.  Our proof of Theorem \ref{thm: Loesers conjecture} starts from the idea of Loeser \cite{Loeser}, in the case of isolated singularities, that one can use the polarization of the mixed Hodge structure on the Milnor fiber of $f$ to study $Z_f$. However, this strategy does not seem to work for non-isolated singularities, as many inputs there rely crucially on the isolatedness assumption: Varchenko's theory of asymptotic mixed Hodge theory, the compactification of the Milnor fiber of isolated singularities of J. Scherk, Barlet's pairing, Deligne's splitting of mixed Hodge structures, to name a few. 

We overcome the difficulty by instead relating $Z_f$ to the theory of (polarized) complex Hodge modules of Sabbah and Schnell \cite{MHMproject}; the necessary elements of this theory are recalled in \S\ref{sec:MHM}. It is known (see e.g., \cite{SY23}) that the minimal exponent can be detected by the lowest piece of the Hodge filtration on the Hodge module of vanishing cycles for the function $f$. Moreover, we observe that the order and residue (of the highest order) of the associated pole of $Z_f$ at $-\tilde{\alpha}_f$ are controlled by the weight filtration and polarization on the vanishing cycles, respectively (see \S\ref{sec:zeta vs polarization} and Theorem \ref{thm: sheaf theoretic loeser}).

The key ingredient is the following fundamental fact about polarized Hodge modules.

\begin{thm}[Theorem \ref{thm:lowest piece is positive}] \label{thm:intro lowest piece is positive}
Let $\mc{M}$ be a complex polarized Hodge module on a complex manifold $X$. Then the integral of the polarization is either positive or negative definite on the lowest piece of the Hodge filtration in $\mc{M}$.
\end{thm}
A special case of this theorem has been proved in \cite[Proposition 4.7]{DV22}.

As a historical aside, we remark that the theory of polarizations on complex Hodge modules, which follows \cite{Sabbah05}, itself builds on ideas of Barlet and Maire \cite{Barletmellintransform} on poles of zeta functions, see \cite[Page 7]{Sabbah05}.

The methods behind Theorem \ref{thm: Loesers conjecture} also provide information about the other poles of $Z_f$, also extracted from the Hodge filtration on the vanishing cycles. For a holomorphic function $f$ on a complex manifold $X$, we consider the left $\sD$-module $\gr_V^\alpha \iota_+\mc{O}_X$ with $\alpha \in [0, 1)$, where $\iota_+\mc{O}_X = \bigoplus_{i \geq 0} \mc{O}_X \otimes \partial_t^i$ is the direct image of $\mc{O}_X$ under the graph embedding $\iota : X \to X \times \mb{C}$, and $V^\bullet$ denotes the $V$-filtration of Kashiwara--Malgrange along the coordinate $t$ on $\mb{C}$. By a theorem of Kashiwara \cite{Kas83}, the $\sD$-module $\bigoplus_{\alpha \in [0,1)}\gr_V^\alpha \iota_+\mc{O}_X$ maps to the perverse sheaf of vanishing cycles under the Riemann-Hilbert correspondence. Furthermore, Malgrange \cite{Malgrange} showed that if $-\beta$ is a root of $b_f(s)$, then $\gr^{\beta}_V\iota_{+}\cO_X\neq 0$.
These $\sD$-modules are endowed naturally with a \emph{Hodge filtration} $F_\bullet$ defined by\footnote{Following \cite{MHMproject}, our convention is that the Hodge filtration on $\iota_{+}\cO_X$ starts at degree 1, see \eqref{eqn: convention of Hodge filtration on the embedding}.}
\[ F_{k+1}\iota_{+}\cO_X=\sum_{0\leq i\leq k} \cO_X\otimes \d_t^{i}, \quad F_{k}\gr^{\alpha}_V\iota_{+}\cO_X=\frac{F_k\iota_{+}\cO_X\cap V^{\alpha}\iota_{+}\cO_X}{F_k\iota_{+}\cO_X\cap V^{>\alpha}\iota_{+}\cO_X}.\]

\begin{thm} \label{thm: producing poles}
Let $\alpha \in [0, 1)$ such that $b_f(s)$ has a root in $-\alpha+\Z_{<0}$ and let $k$ be the smallest integer with $F_{k+1}\gr^{\alpha}_V\iota_{+}\cO_X\neq 0$. Then $- \alpha - k$ is a pole of $\tilde{Z}_f$ with order at least
\[ \min \{ r\in \Z_{\geq 1}\mid (s+\alpha)^r\cdot F_{k+1}\gr^{\alpha}_V\iota_{+}\cO_X=0\}.\]
If $\alpha + k  < \tilde{\alpha}_f  + 1$, then $-\alpha - k$ is also a root of $\tilde{b}_f(s)$.
\end{thm}
By Theorem \ref{thm: Loesers conjecture} and Corollary \ref{corollary: determine the center of minimal exponent}, the pole order bound above is attained in the case of the minimal exponent, i.e.\ for $k = \lfloor \tilde{\alpha}_f \rfloor$ (Lemma \ref{lem:mink}) and $\alpha = \tilde{\alpha}_f - k$.

It is easy to see that if $\beta$ is a pole of $Z_f$, then $\beta-1$ is also a pole of $Z_f$. Therefore, by the result of Barlet discussed above, to find all poles of $Z_f$, it suffices to find for each root $s_0$ of $b_f(s)$ the smallest shift $m\in \Z_{\geq 0}$ such that $s_0 - m$ is a pole. Theorem \ref{thm: producing poles} controls such $m$ in terms of Hodge theory. We also give a more classical interpretation via Milnor fibers (Proposition \ref{prop: strengthen Barlet}), which strengthens results of \cite{Barlet86}.

\subsection{Hodge and higher multiplier ideals}

Integrals of the form \eqref{eq:intro zeta function} can also be used to define the \emph{multiplier ideals} (e.g. see \cite[Chapter 9]{Lazarsfeld}). These capture finer information than $\lct(f)$ about the failure of $D=\mathrm{div}(f)$ to be log canonical. To extend this to the setting of log canonical singularities, Musta\c{t}\u{a} and Popa \cite{MP16}, respectively Schnell and Yang \cite{SY23}, introduced the \emph{Hodge ideals} $I_k(f^\alpha)$ and \emph{higher multiplier ideals} $\cJ_k(f^{\alpha})$ for $k\geq 0$; the higher multiplier ideals is equivalent to Saito's microlocal $V$-filtration \cite{Saito16}. They specialize to the multiplier ideals when $k = 0$:
\[ I_0(f^{\alpha})=\cJ_0(f^{\alpha-\epsilon})=\mc{J}(f^{\alpha - \epsilon}).\]
These ideals can be defined either in terms of a resolution of singularities, or in terms of the Hodge filtrations on certain mixed Hodge modules.

Unlike multiplier ideals, the ideals $I_k$ and $\mc{J}_k$ have no analytic characterizations in the literature. We provide such when $\alpha + k < \tilde{\alpha}_f + 1$, in which case $I_k(f^\alpha)= \cJ_k(f^{\alpha - \epsilon})$.  When $\alpha + k \geq \tilde{\alpha}_f + 1$, Theorem \ref{thm: analytic charaterization of Vfiltration} below implies another, albeit more complicated analytic description of these ideals, see Remark \ref{remark:analytic description of Hodge ideals and higher multiplier ideals}.

\begin{thm} \label{thm:intro Hodge ideals}
Let $\alpha \in [0, 1)$ and $k \in \mb{Z}_{\geq 0}$ satisfy $\alpha + k < \tilde{\alpha}_f + 1$. Then 
\[ \cJ_k(f^\alpha) = \{g \in \mc{O}_X \mid \text{for all $\varphi$, $\tilde{Z}_f(|g|^2\varphi; s)$ has no poles $\geq -\alpha - k$}\}.\]
\end{thm}
\noindent When $k=0$, Theorem \ref{thm:intro Hodge ideals} describes the multiplier ideal using $Z_f$ as follows
\begin{equation*} \label{eq:intro multiplier ideal}
\mc{J}(f^{\alpha}) = \{g \in \mc{O}_X \mid \text{for all $\varphi$, $Z_f(|g|^2\varphi; s)$ has no poles $\geq -\alpha$}\}.
\end{equation*}

Taking into account also the weight filtrations, \cite{SY23} defined a sequence of ideals 
\[ \cJ_k(f^{\alpha})=W_0 \cJ_k(f^{\alpha - \epsilon}) \subseteq W_1 \cJ_k(f^{\alpha - \epsilon})  \subseteq \cdots \subseteq  W_{\dim X} \cJ_k(f^{\alpha - \epsilon})=\cJ_k(f^{\alpha - \epsilon})\]
interpolating between $\cJ_k(f^{\alpha})$ and $\cJ_k(f^{\alpha - \epsilon})$. If $D$ is reduced and $\alpha=1$, Olano \cite{OlanoweightedHodgeideal} introduces the \emph{weighted Hodge ideals} such that $W_{\bullet}I_k(f)=W_{\bullet}\cJ_k(f^{1-\epsilon})$ if $k<\tilde{\alpha}_f$ and $W_1 \cJ_0(f^{1-\epsilon})$ agrees with the \emph{adjoint ideal}  \cite{Olanoweightedmultiplierideal}. We give a description of these.

\begin{thm} \label{thm:intro weighted ideals}
For $\alpha \in [0, 1)$, $k,\ell \in \mb{Z}_{\geq 0}$ and $\alpha + k < \tilde{\alpha}_f + 1$, we have
\[ W_\ell \cJ_k(f^{\alpha-\epsilon}) = \{ g \in \mc{O}_X \mid \text{for all $\varphi$, $(s + \alpha + k)^\ell \tilde{Z}_f(|g|^2\varphi; s)$ has no poles $\geq - \alpha - k$}\}.\]
\end{thm}

\subsection{The $V$-filtration of Kashiwara--Malgrange}
A well-known result (see e.g. \cite{Sabbah}) describes the $V$-filtration on $\iota_+\mc{O}_X$ in terms of Bernstein--Sato polynomials by
\begin{equation}\label{eqn: V-filtration via b function intro} V^\alpha \iota_+\mc{O}_X = \{w \in \iota_+\mc{O}_X \mid \text{all the roots of $b_w(s)$ are $\leq -\alpha$}\}.\end{equation}
See \S\ref{sec: V filtration} for the definition of the notation $b_w(s)$. Musta\c{t}\u{a}--Popa \cite{Mustataprivate} asked whether there is a similar characterization in terms of the poles of zeta functions. More precisely, denote by $(\cO_X)_f$ the localization of $\cO_X$ along $f$ and denote by $\sigma$ the isomorphism \cite{Malgrange}
\[ \sigma: \iota_+ (\mc{O}_X)_f \overset{\sim}\to (\mc{O}_X)_f[s]f^s,\]
see \eqref{eqn: malgrange isomorphism}. Let $w,w'\in \iota_{+}\cO_X$, and define $Z_{w,w'}$ to be the meromorphic extension of 
\begin{equation*}\label{eqn: zeta function in intro} Z_{w,w'}(\varphi; s) = \int_X \sigma(w)\overline{\sigma(w')} \varphi,\end{equation*}
for a smooth test form $\varphi$, where the complex conjugation fixes $s$; see \S \ref{sec:Archimedean zeta function}.

Bernstein's solution to Gelfand's problem, together with \eqref{eqn: V-filtration via b function intro}, shows that
\[ V^\alpha \iota_+\mc{O}_X \subseteq \{ w \in \iota_+\mc{O}_X \mid \text{$Z_{w,w}$ has no poles $> -\alpha$}\}.\]
In view of our results, it is natural to hope that the above inclusion might be an equality; however, this is false in general (see Remark \ref{remark: Prop 1.5 is false for OX} for a counterexample). Instead, we prove the following result, which we believe to be the optimal statement of this kind.

\begin{thm} \label{thm: analytic charaterization of Vfiltration}
For any $\alpha \in \Q$, we have
\[ V^\alpha \iota_+(\mc{O}_X)_f = \{ w \in \iota_+(\mc{O}_X)_f \mid \forall w' \in \iota_+ (\cO_X)_f, \text{$Z_{w,w'}$ has no poles $> -\alpha$}\}.\]
Moreover, given $w \in V^{\alpha}\iota_{+}(\cO_X)_f$, the maximal pole order at $-\alpha$ among $Z_{w,w'}$ with $w'$ varying in $\iota_+ (\cO_X)_f$ is equal to  $\mult_{s=-\alpha}b_w(s)$. 
\end{thm}
Consequently, one obtains an analytic description of Hodge ideals and higher multiplier ideals in general, see Remark \ref{remark:analytic description of Hodge ideals and higher multiplier ideals}. The proof relies on the non-degeneracy of the polarization on $\gr^{\alpha}_V\iota_{+}\cO_X$, which goes back to Sabbah \cite{Sabbah02}. Note that Theorem \ref{thm: analytic charaterization of Vfiltration} fails if we replace the localization $(\mc{O}_X)_f$ with $\cO_X$ itself, see Remark \ref{remark: Prop 1.5 is false for OX}. 

\subsection{Log resolutions, dual complexes and examples}
Finally, we note that Theorem \ref{thm: Loesers conjecture} has the following consequence for the complexity of the pair $(X,D)$ where $D=\mathrm{div}(f)$. Let $\pi:\tilde{X}\to X$ be a log resolution of $(X,D)$ such that $\pi$ is smooth over $X\setminus D_{\mathrm{sing}}$ and $\mathrm{Supp}(\pi^{\ast}D)=\sum_{i\in I}E_i$ is normal crossing. Write
\[ \pi^{\ast}D=\sum_{i\in I} a_iE_i, \quad K_{\tilde{X}/X}=\sum_{i\in I} k_iE_i.\]
Let $J$ be the set of $j \in I$ for which there exists $\ell\in \Z_{\geq 0}$ such that $\tilde{\alpha}_f=(k_j+1+\ell)/a_j$. Let $\cN(\pi, \tilde{\alpha}_f)$ be the dual complex of $\{E_j\}_{j \in J}$ \cite[Definition 12]{KXu}, which records the intersections among the divisors. By Theorem \ref{thm: Loesers conjecture} and \cite[Theorem 5.4.1]{igusa}, we have:
\begin{corollary}\label{corollary: bound of dual complex}
    $\dim \cN(\pi, \tilde{\alpha}_f)\geq\begin{cases}
        \mult_{s=-\tilde{
    \alpha}_f}b_f(s), \quad &\textrm{if $\tilde{\alpha}_f\in \Z_{\geq 2}$},\\
     \mult_{s=-\tilde{
    \alpha}_f}b_f(s)-1, \quad &\textrm{if $\tilde{\alpha}_f\not\in \Z_{\geq 2}$}.
    \end{cases}$
\end{corollary}
A similar argument shows that $\textrm{mult}_{s=-\mathrm{lct}(f)}b_f(s)$ is equal to the maximal number of the divisors $E_i$ computing the $\mathrm{lct}(f)$ and intersecting nontrivially; see also 
\cite[Lemma 3.2]{AdiceamMarmon}. If the proper transform $\tilde{D}$ of $D$ is smooth, then using the arguments of \cite{Atiyah,BG}, Theorem \ref{thm: Loesers conjecture} recovers the fact that $\tilde{\alpha}_f\geq \min_{E_i\neq \tilde{D}}\frac{k_i+1}{a_i}$ (proved in \cite{MPVfiltration,DirksMustata}). As observed by Koll\'ar \cite[Remark 10.8]{kollar}, the inequality can be strict: from our perspective, this is explained by potential cancellation between poles from different charts of $\tilde{X}$. 

To finish the introduction, we discuss two examples in view of Theorem \ref{thm: Loesers conjecture}.
\begin{example}\label{example: Whitney umbrella}
Let $D=\mathrm{div}(f)$ be the Whitney umbrella, where $f=x_1^2 - x_2^2 x_3$ on $X= \C^3$.
Let $\pi:\tilde{X}\to X$ be the blowup of $X$ along the $x_3$-axis with $\tilde{X} \subset \C^3 \times \mathbb{P}^1$.  Denote by $[y_1, y_2]$ the projective coordinates on $\mathbb{P}^1$ and consider the opens $U_i = \{(x,y) \in \tilde{X} \, \mid \, y_i \neq 0 \} \cong \C^3$ for $i=1,2$. On $U_1$, set $x_2' = y_2/y_1$ and on $U_2$ set $x_1' = y_1/y_2$. Then 
\begin{align*}
(\pi \circ f)_{|U_1} = x_1^2(1-x_2'^2 x_3), &\quad \pi^*(dx)_{|U_1} = x_1 dx'|_{U_1},\\
(\pi \circ f)_{|U_2} = x_2^2(x_1'^2 - x_3), &\quad \pi^*(dx)_{|U_2} = x_2 dx'|_{U_2},
\end{align*}
where $dx$ and $dx'$ stand for the standard volume forms on $X$ and $X'$, respectively.

Let $B \subseteq X$ be a small closed ball centered at $0$, of radius $\varepsilon \ll 1$, and $\varphi$ a non-negative test function supported on $B$ with $\varphi(0)>0$. Consider the open $U_1' = \{(x_1, x_2', x_3) \in U_1 \, \mid \, |x_2'| < \varepsilon\}$ and choose a partition of unity $\{p_1, p_2\}$ subordinate to the open covering $U_1' \cup U_2 = \tilde{X}$. Then
\begin{align*}
Z_f(\varphi; s) 
= \int_{U_1'} |x_1^2|^{2s+1}&|1-x_2'^2x_3|^{2s} \cdot p_1 \cdot (\pi\circ\varphi) \, dx' \wedge d\bar{x'}+  \\
&+\int_{U_2} |x_2^2|^{2s+1}|x_1'^2-x_3|^{2s} \cdot p_2 \cdot (\pi\circ\varphi) \, dx' \wedge d\bar{x'}.
\end{align*}
Notice that on $U_1' \cap \pi^{-1}(B) \supseteq \mathrm{supp}(p_1 \cdot (\pi \circ \varphi))$, we have $|(x_2')^2x_3| < \varepsilon ^3 \ll 1$, and thus $|1-x_2'^2 x_3|^2$ is invertible. It follows that the first summand has pole at $s=-1$ of order at most $1$ since only the exceptional divisor contributes. On the other hand, it is easy to see that the second summand has order $2$. Thus, the pole order of $Z_f$ at $s=-1$ is at least $2$. The same argument for an arbitrary $\varphi$ shows that the pole order is $\leq 2$, so pole order is exactly $2$. In fact, the same argument shows that the pole order of $Z_f$ at any negative integer is $2$. In particular, we conclude $\tilde{\alpha}_f=1$ by Theorem \ref{thm: Loesers conjecture}.
\end{example}

\subsection*{Outline of the paper} We begin in \S\ref{sec: V filtration} with a discussion of Bernstein--Sato polynomials and the $V$-filtration in a general context. In \S\ref{sec:MHM} we recall the basics of complex Hodge modules, give the proof of Theorem \ref{thm:intro lowest piece is positive} and discuss the more general context of polarized Hodge--Lefschetz modules. In \S\ref{sec: polarization on nearby and vanishing cycles} we recall the polarized Hodge--Lefschetz structure on nearby/vanishing cycles and relate it to poles of Archimedean zeta functions and their generalizations. Finally, we apply these tools in \S\ref{sec: proof of main results} to prove the rest of our main results.

\subsection*{Convention}

Throughout the paper, if $\cF$ is a sheaf on a complex manifold $X$, the notation $m\in \cF$ means that $m\in \Gamma(U,\cF|_U)$ for some open subset $U\subseteq X$.

\subsection*{Acknowledgements}

We thank Mihnea Popa, Christian Schnell, Mircea Musta\c{t}\u{a}, Claude Sabbah, Bradley Dirks, and Kari Vilonen for helpful discussions. We thank Raf Cluckers for pointing out the reference \cite{AdiceamMarmon}.

\section{$V$-filtrations and Bernstein--Sato polynomials}\label{sec: V filtration}

Throughout the section, let $f$ be a holomorphic function on a complex manifold $X$. Denote by $\iota: X \to X\times \C$ the graph embedding sending $x$ to $(x,f(x))$. Let $t$ be the coordinate on $\C$ and denote by $s\colonequals-\d_t t$. Let $\cM$ be a regular holonomic \emph{left} $\sD_X$-module with quasi-unipotent monodromy along $f$.

\subsection{$V$-filtrations}\label{sec: definition of V filtration}
The following is proved by Kashiwara \cite{Kas83} and Malgrange \cite{Malgrange}.
\begin{thm}\label{thm: uniqueness of V filtration}
    There exists a unique exhaustive decreasing filtration on $\iota_{+}\cM$, called the \emph{$V$-filtration along $f$}, which satisfies
    \begin{enumerate}
        \item $V^{\bullet}\iota_{+}\cM$ is indexed left-continuously and discretely by rational numbers so that $V^{>\alpha}\iota_{+}\cM=V^{\alpha+\epsilon}\iota_{+}\cM$ for $0<\epsilon\ll 1$,
        \item for each $\alpha\in \Q$, $V^{\alpha}\iota_{+}\cM$ is a coherent $\sD_X\langle s,t\rangle$-module,
\item $t\cdot V^{\alpha}\iota_{+}\cM \subseteq V^{\alpha+1}\iota_{+}\cM$, with equality for $\alpha\gg 0$ and $\d_t\cdot V^{\alpha}\iota_{+}\cM\subseteq V^{\alpha-1}\iota_{+}\cM$,
    \item $s+\alpha$ acts nilpotently on $\gr_V^{\alpha}\iota_{+}\cM \colonequals V^\alpha \iota_+\cM / V^{>\alpha} \iota_+ \cM$.
    \end{enumerate}
\end{thm}

Assume $f$ acts bijectively on $\cM$ and consider the $\sD_X\langle s, t \rangle$-module $\cM[s]f^s$, where 
\begin{align*}
    \xi\cdot(ms^{\ell}f^s)&=\left(\xi(m)s^{\ell}+\frac{\xi(f)}{f}ms^{\ell+1}\right)f^s, \quad \forall \xi\in T_X,\\
    s\cdot (ms^{\ell}f^s)&=ms^{\ell+1}f^s,\quad t\cdot (ms^{\ell}f^s)=\left(fm(s+1)^{\ell} \right) f^s,
\end{align*}
and $st=t(s-1)$. By \cite{Malgrange}, there is an isomorphism of $\sD_X \langle s,t\rangle$-modules
\begin{equation}\label{eqn: malgrange isomorphism}
\sigma:\iota_{+}\cM = \bigoplus_{\ell \geq 0} \cM \otimes \partial_t^\ell \xrightarrow{\sim} \cM[s]f^s, \quad m\otimes \d_t^{\ell} \to \frac{m}{f^{\ell}}\prod_{j=0}^{\ell-1}(-s+j)f^s,
\end{equation}
where the inverse morphism is given by $ms^{\ell}f^s\mapsto m\otimes (-\d_tt)^{\ell}$ (see also \cite[Proposition 2.5]{MPVfiltration}). We freely use the identification for submodules of such $\cM$ as well. 

\subsection{Bernstein-Sato polynomials}
In this section we assume $f$ acts injectively on $\cM$. Since $\mc{M}$ injects into its localization $\mc{M}_f$, we can regard sections of $\iota_+\mc{M}$ as elements in $\mc{M}_f[s]f^s$ via the isomorphism \eqref{eqn: malgrange isomorphism}.
\begin{definition}\label{definition: definition of BS polynomials}

For a local section $w\in \iota_{+}\cM$, the \emph{Bernstein--Sato polynomial} (or \emph{$b$-function}) is the unique monic polynomial $b_w(s)\in\C[s]$ of smallest degree satisfying
\begin{align}\label{eqn: definition of b-function}
     P(s,t)\cdot (t\cdot w)=b_w(s)\cdot w,
\end{align}
for some $P(s,t)\in \sD_X\langle s,t\rangle$.
\end{definition}
By \cite{Kas76}, the $b$-function $b_w(s)$ only exists locally. Up to shrinking the open subset containing the support of $w$, we can assume $b_w(s)$ exists. 
\begin{remark}
If $\cN\subseteq \cM$ is a submodule and $w\in \iota_{+}\cN$, then the two notions of $b_w(s)$ coincide. If $\cM=(\cO_X)_f$, the isomorphism \eqref{eqn: malgrange isomorphism} gives $b_{1\otimes 1}(s)=b_f(s)$ as in the introduction.
\end{remark}

\begin{thm}[\cite{Sabbah}]\label{thm: Sabbah}
For any $\alpha \in \Q$, we have
    \[ V^{\alpha}\iota_{+}\cM=\{w \in \iota_{+}\cM \mid \textrm{ all the roots of $b_w(s)$ are $\leq -\alpha$}\}.\]
\end{thm}

\begin{lemma}\label{lem:easy}
For $w \in \iota_{+}\cM$ and $r \in \C$, then $b_w(s)$ divides $(s+r) \cdot b_{(s+r)w}(s)$.
\end{lemma}

\begin{proof}
Let $P(t,s)\in \sD_X\langle s,t\rangle$ yielding an equation 
\[P(s,t) \cdot (t\cdot (s+r)w) = b_{(s+r)w}(s) \cdot (s+r)w.  \]
The divisibility readily follows from $t(s+r)=(s+r+1)t$.
\end{proof}

For $k\in \Z_{\geq 0}$ and $r\in \C$, we denote $[s+r]_k := \prod_{i=0}^{k-1} (s+r+i)$.

\begin{lemma}\label{lem:multroot}
Let $w \in \iota_{+}\cM$ and $\alpha \in \Q$. Then 
\[\mult_{s=-\alpha}b_{(s+\alpha)w}(s) = \max\left(\mult_{s=-\alpha}b_w(s)-1,0\right).\]
\end{lemma}

\begin{proof}
First, by Lemma \ref{lem:easy}, $b_{w}(s)$ divides $(s+\alpha)b_{(s+\alpha)w}(s)$, so 
\[\mult_{s=-\alpha}b_{(s+\alpha)w}(s) \geq \max\left(\mult_{s=-\alpha}b_w(s)-1,0\right).\]
It suffices to establish the other direction.

By definition, there exists some $P(t,s)\in \sD_X\langle s,t\rangle$ such that
\[ P(s,t)\cdot (tw) = b_w(s) \cdot w .\]
Let $P(s,t) = \sum_{i=0}^{k-1} P_i(s) \cdot t^i$, $P_i(s) \in \sD_X[s]$. Multiplying the above by $[s+\alpha+1]_k$ and noticing that (because $st=t(s-1)$)
\[ [s+\alpha+1]_k t^i(tw)=\frac{[s+\alpha+1]_k}{s+\alpha+1+i}(s+\alpha+1+i)t^{i+1}w=\frac{[s+\alpha+1]_k}{s+\alpha+1+i}t^{i+1}(s+\alpha)w,\]
we get
    \[Q \cdot (t(s+\alpha)w) = [s+\alpha+1]_k \cdot b_w(s) \cdot w.\]
where $Q=\left(\sum_{i=0}^{k-1} P_i(s) \frac{[s+\alpha+1]_k}{(s+\alpha+1+i)} \cdot t^i \right)$.
If $b_w(-\alpha)=0$, then by the minimality we have
\[ b_{(s+\alpha)w}(s) \, \mid \,\, [s+\alpha+1]_k \cdot \frac{b_w(s)}{(s+\alpha)}.\]
If $b_w(-\alpha)\neq 0$, then further multiplying $(s+\alpha)$ gives
\[   \left((s+\alpha)Q\right) \cdot (t(s+\alpha)w) = [s+\alpha+1]_k \cdot b_w(s) \cdot (s+\alpha) w,\]
and so by the minimality $b_{(s+\alpha)w}(s) \, \mid \,\, [s+\alpha+1]_k \cdot b_w(s)$. We conclude that \[\mult_{s=-\alpha}b_{(s+\alpha)w}(s) \leq \max\left(\mult_{s=-\alpha}b_w(s)-1,0\right),\]
and this proves the desired equality.
\end{proof}

\begin{corollary}\label{cor:Nmu}
For every $\alpha\in \mathbb{Q}$ and $\ell \in \Z_{\geq 0}$, we have
    \[ \ker\{(s+\alpha)^\ell:\gr^\alpha_V\iota_+\cM\to \gr^\alpha_V\iota_+\cM\} =\{[w] \in \gr^\alpha_V\iota_+\cM \mid w \in V^{\alpha}\iota_+ \cM, \, \mult_{s=-\alpha}b_{w}(s) \leq \ell\}.\]    
\end{corollary}

\begin{proof}
    It follows from Lemma \ref{lem:multroot} and Theorem \ref{thm: Sabbah}.
\end{proof}
The next result sharpens \cite[Proposition 6.12] {MPVfiltration} and is crucial for Theorem \ref{thm: Loesers conjecture}.
\begin{prop}\label{prop:sharp}
Let $\cM=\cO_X$. For any $\ell \in \Z_{\geq 0}$, one has
\begin{equation} \label{eq:sharp 1}
b_{1\otimes \d_t^{\ell}}(s)= (s+1) \cdot \tilde{b}_{f}(s-\ell).
\end{equation}
\end{prop}
\begin{proof}
We may as well assume $\ell > 0$. Applying $t^\ell$ to $1\otimes \d_t^{\ell}$ and using $b_{t\cdot w}(s)=b_w(s+1)$ for any $w\in\iota_{+}\cO_X$, \eqref{eq:sharp 1} is equivalent to
\[ b_{w_\ell}(s) = \tilde{b}_f(s) (s+ \ell+1) \]
where $w_\ell= (-t)^\ell (1 \otimes \partial_t^\ell) = [s+1]_{\ell} \, \cdot f^s$ under the isomorphism \eqref{eqn: malgrange isomorphism}; here $\cO_X$ injects into the localization $(\cO_X)_f$.  It follows from \cite[Proposition 6.12]{MPVfiltration} that
\[ \tilde{b}_f(s)\mid b_{w_\ell}(s) \mid \tilde{b}_f(s) (s+ \ell+1).\]
It remains to prove 
\[\mult_{s=-\ell-1}b_{w_\ell}(s)\geq \mult_{s=-\ell-1}\tilde{b}_f(s)+1=\mult_{s=-\ell-1}b_f(s)+1.\]
Let $\sD_X[s]_{(s+\ell+1)}$ be the localization of $\sD_X[s]$ at the ideal $(s+\ell+1) \subseteq \C[s]$. Localizing (\ref{eqn: definition of b-function}) for $w_{\ell}$ and expanding $P(s, t)=\sum_i P_i(s)\cdot t^i$ with $P_i(s)\in \sD_X[s]$, we see that $\mult_{s=-\ell-1}b_{w_\ell}(s)$ can be characterized as the smallest $\nu$ such that 
\begin{align*}
(s+\ell+1)^\nu \cdot f^s &\in \sum_{i=1}^\ell \sD_X[s]_{(s+\ell+1)} (s+\ell+1) f^{s+i}  \,\, + \,\, \sum_{j=\ell+1}^\infty \sD_X[s]_{(s+\ell+1)} f^{s+j}\\
&=\sD_X[s]_{(s+\ell+1)} (s+\ell+1) f^{s+1} + \sD_X[s]_{(s+\ell+1)}f^{s+\ell+1}.\end{align*}
We work with local coordinates $x_i$ and partial derivatives $\partial_i$. Consider an equation 
\[(s+\ell+1)^\nu f^s = P \cdot (s+\ell+1)f^{s+1} + Q \cdot f^{s+\ell+1},\]
with $P,Q \in \sD_X[s]_{(s+\ell+1)}$. Write $Q = q(x,s)+ \sum_i Q_i \cdot \d_i$, with $q(x,s) \in \cO_X[s]_{(s+\ell+1)}$ and $Q_i \in \sD_X[s]_{(s+\ell+1)}$. We obtain an equation of the form
\[(s+\ell+1)^\nu f^s = P' \cdot (s+\ell+1)f^{s+1} + q(x,s) \cdot f^{s+\ell+1},\] 
with $P' \in \sD_X[s]_{(s+\ell+1)}$. Evaluating at $s=-\ell-1$, we must have $\nu \geq 1$, as the right-hand side is a holomorphic function (without poles) whereas $f^{-\ell-1}$ is not, and hence $(s+\ell+1) \mid q(x,s)$. Thus, after dividing by $(s+\ell+1)$ we can write
\[(s+\ell+1)^{\nu-1} f^s =P'' \cdot f^{s+1},\]
with $P'' \in \sD_X[s]_{(s+\ell+1)}$. We conclude that $\nu\geq \mult_{s=-\ell-1}b_f(s)+1$.
\end{proof}

\begin{remark}
   Using the proof of the first divisibility in \cite[Proposition 6.12]{MPVfiltration}, one can get an explicit operator yielding the $b$-function of $1\otimes \d_t^{\ell}$ if one knows such for $b_f(s)$. After we posted the first version of this paper on \texttt{arXiv},  Musta\c{t}\u{a} pointed out that Proposition \ref{prop:sharp} is also proved in  \cite[Thoerem 10.17]{Vfiltrationintroduction}.
\end{remark}

\section{Complex Hodge modules and polarizations} \label{sec:MHM}

In this section, we briefly review the theory of complex Hodge modules from \cite[\S14]{MHMproject} and study the positivity properties of the polarization on the lowest piece of the Hodge filtration.

\subsection{Polarized Hodge modules}

The basic objects in the theory of Hodge modules are the \emph{polarized Hodge modules}. A complex polarized Hodge module on a complex manifold $X$ consists of the following four pieces of data:
\begin{enumerate}
	\item A regular holonomic left $\sD_X$-module $\cM$.
	\item An increasing filtration $F_{\bullet} \cM$, called the \emph{Hodge filtration}, by coherent $\sO_X$-submodules.
		This filtration is assumed to be good, which means that: it is exhaustive; $F_k
		\cM = 0$ for $k \ll 0$ locally on $X$; and one has $F_k \sD_X \cdot F_{\ell}
		\cM \subseteq F_{k+\ell} \cM$, with equality for $k \gg 0$ locally on $X$.
	\item A pairing $S:\cM \otimes_{\C} \overline{\cM} \to
		\Db_X$, called the \emph{polarization}, valued in the sheaf of distributions on $X$. Here $S$ is assumed to be Hermitian and $\sD_X$-linear in its first argument (and
		therefore conjugate linear in its second argument).
	\item An integer $w \in \Z$, called the \emph{weight}.
\end{enumerate}
The tuple $(\cM, F_{\bullet} \cM, S, w)$ is said to be a \emph{polarized Hodge module} if it satisfies
several additional conditions that are imposed on the nearby and vanishing cycle
functors with respect to holomorphic functions on open subsets of $X$ \cite[Definition 14.2.2]{MHMproject}. Note that the weight only enters into this definition through a sign condition for the polarization $S$: we have that $(\cM, F_\bullet \cM, S, w)$ is a polarized Hodge module if and only if $(\cM, F_\bullet \cM, (-1)^w S, 0)$ is one.

\begin{remark}
Since we work here with \emph{left} $\sD$-modules, the sheaf $\Db_X$ of distributions is defined to be dual to the cosheaf of compactly supported smooth top forms; with this definition $\Db_X$ is a left module over $\sD_X \otimes \overline{\sD}_X$. This is a slightly different notion than in the introduction, where we considered distributions as dual to compactly supported smooth functions.
\end{remark}

To illustrate the conventions, let us spell out the definition when $X$ is a point. In this case, a regular holonomic $\sD_X$-module is nothing but a finite dimensional vector space $V$, and a distribution-valued Hermitian form is simply a Hermitian form on $V$ in the usual sense. By definition, a tuple $(V, F_\bullet V, S, w)$ defines a polarized Hodge module on a point (a.k.a., a polarized Hodge structure) if and only if
\[ V = \bigoplus_{p + q = w} V^{p, q} \quad \text{where} \quad V^{p, w - p} = F_{-p} V \cap (F_{-p - 1} V)^\perp\]
and $S|_{V^{p, q}}$ is $(-1)^q$-definite for all $p, q$. Here the orthogonal complement is taken with respect to $S$.

Let us briefly comment on the relation between this complex theory and Saito's original notion of polarized Hodge module \cite{Saito88} with $\Q$-coefficients; we refer the reader to \cite[Appendix]{DV22} for a more detailed explanation. For Saito, polarized Hodge modules are defined to be certain special tuples $(\mc{M}, F_\bullet \mc{M}, K, S_{\mb{Q}}, \sigma)$, where $\mc{M}$ and $F_\bullet \mc{M}$ are as above, $K$ is a $\mb{Q}$-perverse sheaf, $S_{\mb{Q}} \colon K \otimes K \to \mb{Q}_X[n]$ is a (skew-)symmetric bilinear form and $\sigma \colon K \otimes_{\mb{Q}}\mb{C} \cong \mathrm{DR}(\mc{M})$ is an isomorphism of complex perverse sheaves. In the complex version of the theory, one writes the rational polarization $S_\mb{Q}$ as the real or imaginary part of a unique Hermitian form
\[ S_\mb{C} \colon (K \otimes_{\mb{Q}} \mb{C}) \otimes \overline{(K \otimes_{\mb{Q}} \mb{C})} \to \mb{C}_X[n], \]
which may then be identified with a Hermitian form
\[ S \colon \mc{M} \otimes \overline{\mc{M}} \to \Db_X\]
via Kashiwara's isomorphism \cite{Kas86}: \[\mb{D}\overline{\mathrm{DR}(\mc{M})} \cong \mathrm{DR}\left(\shom_{\overline{\sD}_X}(\overline{\mc{M}}, \Db_X)\right).\]

One of the main ideas of \cite{MHMproject} is that, from this perspective, the $\mb{Q}$-structure of $K$ plays only an auxiliary role, and the entire theory of Hodge modules can be developed without it.

\begin{remark}
Beyond the emphasis on perverse sheaves versus $\sD$-modules, the main technical difference between $\mb{Q}$-Hodge modules as defined by Saito and complex Hodge modules as defined by Sabbah-Schnell is that the former must always have quasi-unipotent monodromy (since this is true for any polarized variation of $\mb{Q}$-Hodge structure), while the latter need only have monodromy whose generalized eigenvalues are of absolute value $1$: this amounts to allowing $V$-filtrations indexed by $\mb{R}$ instead of $\mb{Q}$. Incorporating such $\sD$-modules into the theory requires a generalization of Schmid's Nilpotent Orbit Theorem: see \cite{sabbah2022degeneratingcomplexvariationshodge} for details. This greater generality is not needed for our purposes, however, since all $\sD$-modules appearing in this paper are built functorially from $\mc{O}_X$ and hence have quasi-unipotent monodromy.
\end{remark}

Despite the complicated definition, polarized Hodge modules turn out to be very natural objects: by the Structure Theorem \ref{thm: structure theorem} below, they are precisely the intermediate extensions of polarized variations of Hodge structure on locally closed analytic subsets. To state the theorem precisely, we first recall what this means.

First, suppose that $i : Z \to X$ is the inclusion of a smooth closed analytic subset $Z \subseteq X$ and suppose that $(\mc{V}, F_\bullet \mc{V}, S)$ is a polarized complex variation of Hodge structure on $Z$. Recall that the pushforward of the $\sD_Z$-module $\mc{V}$ is defined by
\[ i_+ \mc{V} = i_*(\sD_{X \leftarrow Z} \otimes_{\sD_Z} \mc{V}),\]
where $\sD_{X \leftarrow Z} = \mathrm{Diff}(i^{-1}\omega_X, \omega_Z) \cong i^{-1}\sD_X \otimes_{i^{-1}\mc{O}_X} \omega_{Z/X}$. In particular, $i_+\mc{V}$ is generated over $\sD_X$ by the subsheaf $i_*(\mc{V} \otimes \omega_{Z/X})$. We endow $i_+\mc{V}$ with a Hodge filtration by
\begin{equation}\label{eqn: convention of Hodge filtration on the embedding} F_p i_+\mc{V} \colonequals \sum_{p_1 + p_2  + c = p } F_{p_1} \sD_X \cdot i_*(F_{p_2} \mc{V} \otimes \omega_{Z/X}),\end{equation}
where $c = \codim_X Z$. Note that the shift by $c$ is a feature of the Hodge filtration on left Hodge modules. Similarly, we endow $i_+\mc{V}$ with a polarization  $i_+S$ defined by the formula
\begin{equation} \label{eq:polarization pushforward}
 \langle \varphi, i_+S(v \otimes \xi, v' \otimes \xi') \rangle = (-1)^{c(c - 1)/2}(2 \pi \sqrt{-1})^c\int_Z (\xi \contract \overline{\xi'} \contract \varphi)|_Z \cdot S(v, v')
\end{equation}
for a smooth compactly supported top form $\varphi$ on $X$, and $v \otimes \xi$, $v'\otimes \xi' \in i_*(\mc{V} \otimes \omega_{Z/X})$, extended by linearity to $i_+\mc{V}$. Here we write $\langle -,-\rangle$ for the pairing between forms and distributions and $\contract$ for the interior product. This defines a tuple $(i_+\mc{V}, F_\bullet i_+\mc{V}, i_+S)$ on $X$.

Let us briefly comment on the sign in \eqref{eq:polarization pushforward}. Let us say that $\varphi$ is a \emph{local volume form} if at every point $x \in X$, $\varphi$ is either zero or a volume element on the tangent space compatible with the canonical orientation; here we have fixed a square root of $-1$ and hence an orientation on the complex manifold $X$. The sign in \eqref{eq:polarization pushforward} is arranged so that if $\varphi$ is a local volume form on $X$ and $\xi \in \omega_{Z/X}$ is a holomorphic section then
\[ (-1)^{c(c-1)/2} (2\pi \sqrt{-1})^c (\xi \contract \overline{\xi} \contract \varphi)|_Z\]
is a local volume form on $Z$. 

Let us now suppose that $Z \subseteq X$ is an arbitrary closed analytic subset (not necessarily smooth) and let $U \subseteq Z$ be a dense, smooth, Zariski-open subset. Write $W = Z \setminus U$ and $i : U \to X \setminus W$ for the inclusion into $X \setminus W$. We say that a polarized Hodge module $(\mc{M}, F_\bullet \mc{M}, S)$ is an \emph{intermediate extension} of a polarized variation of Hodge structure $(\mc{V}, F_\bullet \mc{V}, S_U)$ on $U$ if $\mc{M}$ has strict support $Z$ (i.e., every sub or quotient of the $\sD_X$-module $\mc{M}$ has support $Z$) and
\[ (\mc{M}, F_\bullet \mc{M}, S)|_{X \setminus W} = (i_+\mc{V}, F_\bullet i_+\mc{V}, i_+S_U).\]
Note that while the polarization $S$ is uniquely determined by its restriction to $X \setminus W$ by \cite[Corollary 12.5.41]{MHMproject}, there are generally many ways to extend $F_\bullet i_+\mc{V}$ to a filtration on $\mc{M}$; the assumption that $(\mc{M}, F_\bullet\mc{M}, S)$ is a polarized Hodge module, however, fixes this extension uniquely. The last assertion follows from the strict $\mathbb{R}$-specializability of $(\cM,F_{\bullet}\cM)$ along any locally defined holomorphic function \cite[Definition 7.2.19 and Definition 9.4.1]{MHMproject}, see also \cite[Exercise 11.2]{Schnelloverview}.

The following is one of the main theorems of the theory of complex Hodge modules \cite[Theorem 16.2.1]{MHMproject}, which is the complex version of \cite[Theorem 0.2]{Saito90}.

\begin{thm}[Structure theorem]\label{thm: structure theorem}
In the setting above, every polarized variation of Hodge structure $(\mc{V}, F_\bullet \mc{V}, S_U)$ of weight $w - \dim Z$ on $U$ has a unique intermediate extension to a polarized Hodge module of weight $w$ on $X$. Moreover, every polarized Hodge module of weight $w$ on $X$ with strict support $Z$ is obtained in this way for some $U$, $\mc{V}$.
\end{thm}

If $j : U \to X$ denotes the inclusion into $X$, we write $j_{!*}(\mc{V}, F_\bullet\mc{V}, S_U) = (j_{!*}\mc{V}, F_\bullet j_{!*}\mc{V}, j_{!*}S_U)$ for the unique intermediate extension.

In the case $U = X$, Theorem \ref{thm: structure theorem} says in particular that every polarized variation of Hodge structure $\mc{V}$ of weight $w - \dim X$ on $X$ is a polarized Hodge module of weight $w$. This is a highly non-trivial statement even in the case $\mc{V} = \mc{O}_X$; the proof relies on, for example, the Hodge--Zucker Theorem, the Nilpotent Orbit Theorem and many delicate arguments due to Saito. 

We now state and prove the precise form of Theorem \ref{thm:intro lowest piece is positive} from the introduction.

\begin{thm} \label{thm:lowest piece is positive}
Let $(\mc{M}, F_\bullet \mc{M}, S)$ be a polarized Hodge module of weight $w$ on $X$. Assume that $F_{p - 1}\mc{M} = 0$ and let $m\in F_p\mc{M}$. Then for any local volume form $\varphi$, we have
\begin{equation} \label{eq:lowest piece is positive 1}
 (-1)^{p + w - \dim X}\langle \varphi, S(m, m) \rangle \geq 0,
\end{equation}
with equality if and only if $\varphi$ vanishes identically on the support of $m$.
\end{thm}

\begin{proof}
By \cite[Theorem 14.2.19]{MHMproject} cf., \cite[\S 5.1.6]{Saito88}, we may decompose $(\mc{M}, F_\bullet\mc{M}, S)$ as a direct sum of polarized Hodge modules with strict support. Since the sign $(-1)^{p + w - \dim X}$ is uniform across all summands, we may therefore assume without loss of generality that $\mc{M}$ itself has strict support $Z \subseteq X$. Then by Theorem \ref{thm: structure theorem}, $(\mc{M},S)=j_{!\ast}(\cV,S_U)$ where $(\cV,S_U)$ is a polarized complex VHS on a smooth Zariski-open subset $U\subseteq Z$ and $j:U\hookrightarrow X$ is the inclusion. Shrinking $U$ and $X$ if necessary (note that the statement is local on $X$), we may assume that $U = (X \setminus D) \cap Z$ for some divisor $D = f^{-1}(0) \subseteq X$ given by the zero locus of a holomorphic function $f : X \to \mb{C}$. 

Consider the graph embedding $\iota : X \to X \times \mb{C}$ of $f$. Since $m \in F_p\mc{M}$ is in the lowest piece of the Hodge filtration, we have that
\[ m \otimes 1 \in F_{p + 1}\iota_+\mc{M}\]
also lies in the lowest piece. It follows from \cite[3.2.1]{Saito88} and the fact that $\mc{M}$ is the intermediate extension of $\mc{V}$ that the lowest piece of the Hodge filtration on $\iota_+\mc{M}$ satisfies
\[ F_{p + 1}\iota_+\mc{M} = V^{>0}\iota_+\mc{M} \cap j'_*(F_{p + 1}\iota|_{X \setminus D, +}(i_+\mc{V})) \subseteq V^{>0}\iota_+\mc{M},\]
where $i : U \to X \setminus D$ and $j' : X \setminus D \to X$ are the inclusions. Hence, the $b$-function $b_m(s)\colonequals b_{m\otimes 1}(s)$ has no roots $\geq 0$ by Theorem \ref{thm: Sabbah}.

To prove the theorem, we need to compute the distribution $S(m, m)$. First, observe that when restricted to any fixed compact support, the distribution $S(m, m)$ will have some finite order $k$, i.e.\ it makes sense to test $S(m, m)$ against $C^k$ test forms with this support. Since the function $|f|^{2s}$ and its $s$-derivative $2\log |f| \cdot |f|^{2s}$ are $C^k$ for $2 \Re s > k$, for any smooth test form $\varphi$ on $X$ the function
\[ \langle \varphi, Z_{m, m}(s)\rangle \colonequals \langle \varphi|f|^{2s}, S(m, m) \rangle \]
is therefore well-defined and holomorphic for $2 \Re s > k$. Note also that if all partial derivatives of $\varphi$ up to order $j$ vanish along $D$, then $\varphi|f|^{2s}$ and its $s$-derivative are $C^k$ for $2 \Re s > k - j$, so $\langle \varphi, Z_{m, m}(s)\rangle$ extends to a holomorphic function on this larger domain.

By Definition \ref{definition: definition of BS polynomials} and the isomorphism \eqref{eqn: malgrange isomorphism}, 
we can choose $P(s) \in \sD_X[s]$  such that 
\begin{equation}\label{eqn: functional eqn of bm(s)} P(s) \cdot (mf^{s + 1}) = b_m(s)(mf^s).\end{equation}
We claim that for $\Re s\gg 0$, we have
\begin{equation}\label{eqn: mero ext of Zmm}
b_m(s) \cdots b_m(s + k) \langle \varphi, Z_{m, m}(s) \rangle = \langle \varphi P(s) \cdots P(s + k)f^{k + 1} |f|^{2s}, S(m, m)\rangle.
\end{equation}
More generally, if $Q(s) \in \sD_X[s]$ is an operator of order $\leq \ell$, i.e.\ $Q(s)\in (F_{\ell}\sD_X)[s]$, we claim that for $\Re s \gg 0$
\begin{equation}\label{eqn: action of Q(s)} \langle \varphi Q(s)|f|^{2s}, S(m, m)\rangle = \sum_i s^i\langle \varphi|f|^{2(s - \ell)}, S(f^\ell m_i, f^\ell m) \rangle,\end{equation}
where we write
\[ Q(s)mf^s = \sum_i s^i m_i f^s\]
and note that $f^\ell m_i \in \mc{M}$ since $Q(s)$ has order at most $\ell$. Note that if \eqref{eqn: action of Q(s)} holds for operators $Q(s)$ and $R(s)$, then it holds also for $Q(s)R(s)$, so it suffices to check the claim when $Q(s)$ has order $\leq 1$. When $Q(s)$ is a holomorphic function this is obvious, so we suppose $Q(s) = \xi$ is a vector field. Recall from \S \ref{sec: definition of V filtration} that
\[ \xi(mf^s)=\xi(m)f^s+s\frac{\xi(f)}{f}mf^s.\]
Then by $\sD_X \otimes \overline{\sD_X}$-linearity of $S$, we have
\begin{align*}
 \langle \varphi \xi |f|^{2s}, S(m, m) \rangle &= \langle \varphi |f|^{2s}\xi + s\varphi |f|^{2(s - 1)}\xi(f) \bar{f}, S(m, m) \rangle \\
 &= \langle \varphi|f|^{2(s - 1)}, S\left(f\xi(m), f m\right)\rangle  + s\left\langle \varphi |f|^{2(s - 1)}, S\left(f\frac{\xi(f)}{f}m, f m\right)\right\rangle.
\end{align*}
Therefore, the claimed equality \eqref{eqn: action of Q(s)} holds. Now we apply the claim to the operator $Q(s) = P(s)P(s + 1) \cdots P(s + k) f^{k + 1}$. Assume the order of $Q(s)$ is $\leq \ell$. From \eqref{eqn: functional eqn of bm(s)}, 
\begin{equation}\label{eqn: Q on mfs} Q(s)(mf^s)=b_m(s) \cdots b_m(s + k)(mf^s).\end{equation}
Then \eqref{eqn: action of Q(s)} implies that for $\Re s\gg 0$,
\begin{align*}
\langle \varphi P(s) \cdots P(s + k)f^{k + 1} |f|^{2s}, S(m, m) \rangle 
&= \langle \varphi |f|^{2(s - \ell)}, S(f^\ell b_m(s) \cdots b_m(s + k) m, f^\ell m) \rangle \\
&= b_m(s) \cdots b_m(s + k) \langle \varphi |f|^{2s}, S(m, m)\rangle.
\end{align*}
This proves \eqref{eqn: mero ext of Zmm}.

Now, rearranging \eqref{eqn: mero ext of Zmm}, we have
\[ \langle \varphi, Z_{m, m}(s) \rangle = \frac{1}{b_m(s) \cdots b_m(s + k)} \langle \varphi P(s) \cdots P(s + k)f^{k + 1} |f|^{2s}, S(m, m) \rangle.\]
Since every coefficient of $s$ in $\varphi P(s) \cdots P(s + k)f^{k + 1}$ has derivatives vanishing to order at least $k + 1$ along $D$, and since $b_m(s)$ has no roots with real part $\geq 0$, the right hand side is well-defined and holomorphic in the region $\{s \in \mb{C} \mid \Re s \geq 0\}$. Hence, we may analytically continue $\langle \varphi, Z_{m, m}(s)\rangle$ to a holomorphic function on this region satisfying
\begin{align*} 
\langle \varphi, Z_{m, m}(0) \rangle &= \frac{1}{b_m(0) \cdots b_m(k)}\langle \varphi P(0) \cdots P(k) f^{k + 1}, S(m, m)\rangle\\
&=\frac{1}{b_m(0) \cdots b_m(k)}\langle \varphi, S(P(0) \cdots P(k) f^{k+1}m,  m)\rangle\\
&=\langle \varphi, S(m,m) \rangle.
\end{align*}

Now we write out the function $\langle \varphi, Z_{m, m}(s)\rangle$ more explicitly. Fix a smooth bump function $\psi \colon \mb{C} \to [0, 1]$ with compact support such that $\psi(t) = 1$ in a neighborhood of $t = 0$. Then for $2 \Re s > k$, we have
\[ \left(1 - \psi(nf)\right)|f|^{2s} \varphi \to |f|^{2s} \varphi \quad \text{as $n \to \infty$}\]
in the $C^k$-topology. Thus, assuming as above that the distribution $S(m, m)$ is of order $\leq k$ on the support of $\varphi$, and using that $1 - \psi(nf)$ vanishes near $D$ and that $S|_{X\setminus D}=\iota_{+}S_U$, we obtain, for $2 \Re s > k$,
\begin{align*}
 \langle \varphi, Z_{m, m}(s)\rangle &= \langle \varphi|f|^{2s}, S(m, m) \rangle \\
 &=\lim_{n \to \infty} \langle (1 - \psi(nf))|f|^{2s} \varphi, S(m, m)\rangle\\
 &= \lim_{n \to \infty} \langle (1 - \psi(nf))|f|^{2s} \varphi, (i_+S_U)(m|_{X\setminus D}, m|_{X\setminus D}) \rangle \\
&= \lim_{n \to \infty} \int_U (1 - \psi(nf))|f|^{2s}\varphi_U,
\end{align*}
where $\varphi_U$ is the top form on $U$ defined by
\[ \varphi_U = (-1)^{c(c - 1)/2}(2\pi\sqrt{-1})^c S_U(v, v)(\xi \contract \bar{\xi} \contract\varphi)|_U\]
in a local chart where $m = v \otimes \xi$ for some local sections $v \in F_{p - c} \mc{V}$ over $U$ and $\xi \in \omega_{U/X}$. Here we apply \eqref{eq:polarization pushforward}, using the fact that $m$ is a local section of $F_p\cM$, the lowest piece of the Hodge filtration, together with basic properties of the interior product.

Since $v$ lies in the lowest piece $F_{p - c}\mc{V}$ of the Hodge filtration of the polarized variation of Hodge structure $\mc{V}$ of weight $w - \dim X + c$, we have that $(-1)^{p + w - \dim X} S_U(v, v) \geq 0$, with equality only where $v = 0$. By the set-up in \eqref{eq:polarization pushforward}, $(-1)^{c(c - 1)/2}(2\pi\sqrt{-1})^c\left(\xi\contract \bar{\xi}\contract \varphi\right)|_U$ is a local volume form on $U$. So $(-1)^{p + w - \dim X} \varphi_U$ is a local volume form (possibly with non-compact support), which is zero if and only if $\varphi$ vanishes identically on the support of $m$. Since $0 \leq 1 - \psi(nf) \leq 1$ for all $n$ and $1 - \psi(nf) \to 1$ outside a set of measure $0$, Fatou's lemma and the dominated convergence theorem imply that for $2 \Re s > k$,
\[ \int_U |f|^{2s}\varphi_U=\lim_{n \to \infty} \int_U (1 - \psi(nf))|f|^{2s}\varphi_U=\langle \varphi, Z_{m, m}(s)\rangle \]
is a convergent integral since the limit exists.

Since $\langle \varphi, Z_{m, m}(s)\rangle$ has no poles with real part $\geq 0$, we conclude by Lemma \ref{lem:convergence} that
\[(-1)^{p + w - \dim X} \langle \varphi, S(m, m) \rangle = (-1)^{p + w - \dim X} \langle \varphi, Z_{m, m}(0) \rangle = \int_U (-1)^{p + w - \dim X}\varphi_U \]
is a convergent integral. The statement of the theorem follows.\end{proof}

\begin{lemma} \label{lem:convergence}
Let $U$ be a complex manifold, $f : U \to \mb{C}^\times$ a holomorphic function, and $\psi$ a local volume form on $U$, not necessarily with compact support. Suppose that the integral
\begin{equation} \label{eq:convergence 1}
I(s) = \int_U |f|^{2s}\psi
\end{equation}
converges for $\Re s \gg 0$ and admits a meromorphic continuation to $s \in \mb{C}$. Let
\[ s_0 = \max\{ \Re a \mid \text{$s = a$ is a pole of $I(s)$}\}.\]
Then the integral \eqref{eq:convergence 1} is absolutely convergent if and only if $\Re s > s_0$.
\end{lemma}
\begin{proof}
The only if part is clear; we prove the other direction. First, if $|f| \geq C > 0$, then the function $C^{-2s}|f|^{2s}$ is increasing in $\Re(s)$, so the assumptions imply that $I(s)$ converges for all $s$. So, partitioning $U$ if necessary, we may assume that $|f| < 1$. For $r > 0$, let
\[ I_r(s) = \int_{|f| > \frac{1}{r}} |f|^{2s} \psi.\]
This integral converges for all $s$ to a holomorphic function on $\mb{C}$, and we have
\[ I(s) = \lim_{r \to \infty} I_r(s) \quad \text{for $\Re s \gg 0$}.\]
Fix $z \in \mb{R}_{>s_0}$. Since $||f|^{2s}| = |f|^{2 \Re s}$, it suffices to show that the above limit exists for $s = z$. Choose any $s_1 \in \mb{R}$ such that \eqref{eq:convergence 1} converges in a neighborhood of $s = s_1$. Write
\[ I(s) = \sum_{n = 0}^\infty a_n (s_1 - s)^n \quad \text{and} \quad I_r(s) = \sum_{n = 0}^\infty a_n(r) (s_1 - s)^n \]
for the Taylor series expansions. Note that, by assumption, $s = z$ is inside the radius of convergence for both series. By the Cauchy integral formula, we have $a_n(r) \to a_n$ as $r \to \infty$. Moreover, differentiating  $r$ times with respect to $s$ we have 
\begin{equation} \label{eq:convergence 2}
 a_n(r) = \frac{1}{n!}\int_{|f|> \frac{1}{r}}|f|^{2s_1}(-\log |f|^2)^n \psi.
\end{equation}
Since $|f| < 1$, the integrand of \eqref{eq:convergence 2} is non-negative, so the function $a_n(r)$ is increasing. Hence,
\[ a_n(r) \leq a_n \quad \text{for all $r$ and all $n$}.\]
Thus, we have
\[ I_r(z) \leq \sum_{n = 0}^\infty a_n(r)|s_1 - z|^n \leq \sum_{n = 0}^\infty a_n|s_1 - z|^n < \infty, \]
as $z$ is inside the radius of convergence. Since $I_r(z)$ is an increasing function of $r$, we conclude that $\lim_{r \to \infty} I_r(z)$ exists and hence the integral \eqref{eq:convergence 1} converges as claimed.
\end{proof}

\subsection{Polarized Hodge-Lefschetz modules}

Of particular interest to us are the Hodge module structures on nearby and vanishing cycles (or, in other words, on $\sD$-modules of the form $\gr_V^\alpha(-)$). These are not in general polarized Hodge modules, but examples of the following more general structure. For the definition, we recall that if $\op{N}: \mc{M} \to \mc{M}$ is a nilpotent endomorphism of a $\sD_X$-module $\mc{M}$ then the \emph{monodromy weight filtration} $W(\op{N})_\bullet \mc{M}$ is the unique finite increasing filtration such that $N(W(\op{N})_r\mc{M}) \subseteq W(\op{N})_{r - 2}\mc{M}$ for all $r$ and the following is an isomorphism for all $r \geq 0$:
\[ \op{N}^r : \gr^{W(\op{N})}_r \mc{M} \overset{\sim}\to \gr^{W(\op{N})}_{-r} \mc{M}. \]

\begin{definition}\label{definition: polarized HL module}
Let $\mc{M}$ be a regular holonomic $\sD_X$-module, $F_\bullet \mc{M}$ a good filtration, $S$ a distribution-valued Hermitian form on $\mc{M}$ and $\op{N} : \mc{M} \to \mc{M}$ a nilpotent morphism of $\sD_X$-modules. We say $(\cM,F_\bullet \mc{M}, \op{N}, S)$ is a \emph{polarized Hodge-Lefschetz module of central weight $w$} if the following conditions hold.
\begin{enumerate}
\item[(i)] $\op{N}(F_\bullet \mc{M}) \subseteq F_{\bullet + 1}\mc{M}$,
\item[(ii)] $\op{N}$ is self-adjoint with respect to $S$,
\item[(iii)] the morphism $\op{N}^r : (\gr^{W(\op{N})}_{r}\mc{M}, F_\bullet) \to (\gr^{W(\op{N})}_{- r}\mc{M}, F_{\bullet + r})$ is a filtered isomorphism for all $r\geq 0$, and
\item[(iv)] the tuple $(\op{P}\gr^{W(\op{N})}_{r}\mc{M}, F_\bullet, (-1)^rS(-, \op{N}^r-),w+r)$ is a polarized Hodge module for all $r \geq 0$, where 
\[ \op{P} \gr^{W(\op{N})}_r \mc{M} \colonequals \ker(\op{N}^{r + 1} :\gr^{W(\op{N})}_{r} \mc{M} \to \gr^{W(\op{N})}_{- r - 2}\mc{M}).\]
\end{enumerate}
\end{definition}
The morphism $\op{N}$ is automatically strict with respect to $W(\op{N})_\bullet$ by construction and strict with respect to $F_\bullet$ by $(\mathrm{iii})$. Note that self-adjointness of $\op{N}$ implies that \begin{equation}\label{eqn: selfadjointness}
S(W(\op{N})_r\mc{M}, W(\op{N})_{r'}\mc{M}) = 0, \quad \textrm{for $r + r' < 0$},
\end{equation}
so the polarization in $(\mathrm{iv})$ is well-defined. 

We recall in the next section how to regard nearby and vanishing cycles as polarized Hodge-Lefschetz modules; for now, we record some general properties of the lowest piece of the Hodge filtration. In the statement below, we say that a local section $m \in \mc{M}$ has \emph{monodromy weight $r$} if $m \in W(\op{N})_r\mc{M} \setminus W(\op{N})_{r - 1}\mc{M}$. The submodule $\op{P}\gr^{W(\op{N})}_{r}\mc{M}$ is called the \emph{primitive part} and its local sections are called \emph{primitive elements}.

\begin{lemma}\label{lemma: lowest hodge is primitive}
Assume $m\in \cM$ is an element of monodromy weight $r$ such that $0\neq m\in F_p\cM$ and $ F_{p-1}\cM=0$.
Then $r\geq 0$, $[m]\in\gr^{W(\op{N})}_r\cM$ is primitive and 
\begin{equation}\label{eqn: weight level}
r+1=\min\{ \ell\in \Z_{\geq 1} \mid \op{N}^{\ell}\cdot m=0\in \cM\}.\end{equation}
\end{lemma}

\begin{proof}
Clearly $r \geq 0$, since if $r = -k < 0$, we would have
\[ F_{p - k}\gr^{W(\op{N})}_{k} \mc{M} \cong F_p \gr^{W(\op{N})}_{-k} \mc{M} \neq 0\]
contradicting $F_{p - 1}\mc{M} = 0$. Moreover, for $\ell > r + 1$,
\[ F_{p + r + 1}\gr^{W(\op{N})}_{-\ell}\mc{M} = \op{N}^\ell(F_{p + r + 1 - \ell}\gr^{W(\op{N})}_\ell\mc{M}) = 0\]
since $F_{p - 1}\mc{M} = 0$. Hence $\op{N}^{r + 1}m \in F_{p + r + 1} W(\op{N})_{-r - 2}\mc{M} = 0$. This shows that $[m]$ is primitive and, since $\op{N}^r m \neq 0$, \eqref{eqn: weight level} as well.
\end{proof}

\begin{corollary} \label{cor:HL lowest piece is positive}
Let $(\mc{M}, F_\bullet\mc{M}, \op{N}, S)$ be a polarized Hodge-Lefschetz module of weight $w$. Assume that $F_{p - 1}\mc{M} = 0$ and let $m \in F_p\mc{M}$ be a local section of monodromy weight $r$. Then for any local volume form $\varphi$, we have
\[ (-1)^{p + w - \dim X}\langle \varphi, S(m, \op{N}^r m) \rangle \geq 0,\]
with equality if and only if $\varphi$ vanishes identically on the support of $[m]\in \gr^{W(\op{N})}_r\mc{M}$.
\end{corollary}
\begin{proof}
By Lemma \ref{lemma: lowest hodge is primitive}, the class $[m] \in \gr^{W(\op{N})}_r\mc{M}$ is primitive. Since
\[(\op{P}\gr^{W(\op{N})}_r\mc{M}, F_{\bullet}, (-1)^r S(-, \op{N}^r-),w+r)\]
is a polarized Hodge module, by Theorem \ref{thm:lowest piece is positive} we conclude that
\[ (-1)^{p + w + r - \dim X} \langle \varphi, (-1)^r S([m], \op{N}^r [m])\rangle =  (-1)^{p + w - \dim X} \langle \varphi, S(m, \op{N}^r m)\rangle \geq 0. \qedhere\]

\end{proof}

\section{The Archimedean zeta function and polarizations on nearby cycles} \label{sec:zeta vs polarization}

In this section, we establish a new connection between the theory of complex polarized Hodge modules and the poles of Archimedean zeta functions. We fix once and for all a complex manifold $X$ and a holomorphic function $f : X \to \mb{C}$.

\subsection{Polarizations on nearby and vanishing cycles}\label{sec: polarization on nearby and vanishing cycles}

Let $(\mc{M}, F_\bullet\mc{M}, S)$ be a polarized Hodge module on $X$. Recall that, in the theory of $\sD$-modules, the nearby and vanishing cycles of $\mc{M}$ are defined by the formulas
\[
\psi_f\mc{M} \colonequals \bigoplus_{\alpha \in (0, 1]} \psi_{f, e^{-2\pi i\alpha}}\mc{M} \quad \text{and} \quad \phi_f\mc{M} \colonequals \phi_{f, 1}\mc{M}\oplus \bigoplus_{\alpha \in (0, 1)} \psi_{f, e^{-2\pi i\alpha}} \mc{M},
\]
where
\[ \psi_{f, e^{-2\pi i \alpha}}\mc{M} = \gr_V^\alpha \iota_+\mc{M}\quad \text{and} \quad \phi_{f, 1}\mc{M} = \gr_V^0\iota_+\mc{M}.\]
The Hodge filtration on $\iota_+\mc{M}$ induces natural Hodge filtrations on these modules via 
\begin{equation}\label{eqn: shifts of Hodge}
F_\bullet\psi_{f,e^{-2\pi i\alpha}}\mc{M}=F_{\bullet}\gr_V^\alpha \iota_+\mc{M} \quad \text{and} \quad F_\bullet\phi_{f, 1}\mc{M} = F_{\bullet+1}\gr^0_V\cM.
\end{equation}
Similarly, we have the nilpotent operators
\[ \op{N} \colonequals s + \alpha = \alpha - \partial_tt : (\gr_V^\alpha\iota_+\mc{M}, F_\bullet) \to (\gr_V^\alpha\iota_+\mc{M}, F_{\bullet + 1}).\]
The triples $(\gr_V^\alpha \iota_+\mc{M}, F_\bullet, \op{N})$ can be completed to polarized Hodge-Lefschetz modules; we recall the construction of the polarization when $\alpha >0$.

Choose a smooth function $\eta : \mb{C} \to [0, 1]$ with compact support that is identically equal to $1$ near the origin. For $w_1, w_2 \in \iota_+\mc{M}$ and a smooth test form $\varphi$ on $X$, consider the meromorphic function
\[ F_{w_1, w_2}(\varphi; s) = \langle |t|^{2s}\eta(t - f) \varphi \wedge d\mathrm{Vol}_{\mb{C}}, \iota_+S(w_1, w_2)\rangle,\]
where
\[ d \mathrm{Vol}_{\mb{C}} = \frac{\sqrt{-1}}{2\pi} dt \wedge d\bar{t} \]
is a standard volume form on $\mb{C}$: by \cite[\S 12.5.b]{MHMproject}, the function $F_{w_1, w_2}(\varphi;s)$ is well-defined, holomorphic for $\Re(s) \gg 0$ and admits a meromorphic continuation to $s \in \mb{C}$. The polarization $\gr_V^\alpha(S)$ on the nearby cycles $\gr_V^{\alpha}\iota_{+}\cM$ is defined by \cite[12.5.10]{MHMproject}
\begin{equation}\label{eqn: definition of polarization on grValpha}
\langle \varphi, \gr_V^\alpha(S)([w_1], [w_2]) \rangle \colonequals \Res_{s = -\alpha} F_{w_1, w_2}(\varphi; s). \end{equation}
With this definition, $(\gr_V^{\alpha}\iota_{+}\cM, F_{\bullet}, \op{N}, \gr_V^\alpha(S))$ is a polarized Hodge-Lefschetz module of central weight $w - 1$ for all $\alpha \in (0, 1]$ \cite[14.2.2, 12.7.17, 12.5.15]{MHMproject}, cf., \cite[5.2.10.2]{Saito88}. We refer to \cite[Appendix]{DV22} for a proof that this analytic construction agrees with the topological construction given by Saito under the Riemann-Hilbert correspondence.

One can also construct a Hermitian form $\gr_V^0(S)$ on the vanishing cycles module $\gr_V^0\iota_{+}\cM$ such that $(\gr_V^0\iota_{+}\cM, F_{\bullet}, \op{N}, \gr_V^0(S))$ is a polarized Hodge-Lefschetz module of weight $w - 1$\footnote{Note that since the Hodge filtration on $\gr_V^0\iota_{+}\cM$ is shifted by $1$ relative to $\phi_{f, 1}\mc{M}$, the weight is $w - 1$ instead of $w$.} \cite[14.2.2, 12.5.24]{MHMproject}, cf., \cite[5.2.10.3]{Saito88}. We will not need to know the construction, only the following standard properties.

\begin{prop}[{\cite[12.5.28]{MHMproject}}]\label{prop: relation between polarization grVo and grV1}
The morphisms
\begin{align*}
\op{can} \colonequals -\partial_t : (\gr_V^1\iota_{+}\mc{M}, F_\bullet) &\to (\gr_V^0\iota_{+}\cM, F_{\bullet + 1}),\\
\op{var} \colonequals t : (\gr_V^0\iota_{+}\cM, F_\bullet) &\to (\gr_V^1\iota_{+}\mc{M}, F_\bullet), 
\end{align*}
satisfy $\op{can} \circ \op{var} = \op{N}$, $\op{var} \circ \op{can} = \op{N}$ and
\[ \gr_V^0(S)(m, \op{can} m') = -\gr_V^1(S)(\op{var} m, m').\]
\end{prop}

\subsection{Archimedean zeta functions}\label{sec:Archimedean zeta function}
Let us now relate the above discussion to Archimedean zeta functions of the kind considered in the introduction. For elements $w_1, w_2 \in \iota_+(\mc{O}_X)_f$
of the pushforward of the localization of $\cO_X$ along $f$, define a meromorphic distribution $Z_{w_1, w_2}$ on $X$ as follows. Under the isomorphism \eqref{eqn: malgrange isomorphism}, we may write
\[ \sigma(w_1) = \sum_{i=0}^{k_1} m_i s^i f^s, \quad \sigma(w_2) = \sum_{j=0}^{k_2} m_j's^j f^s\]
in $(\cO_X)_f[s]f^s$, with $m_i, m'_i \in (\cO_X)_f$. For a (compactly supported) test form $\varphi$, we set
\begin{equation} \label{eqn:zeta function for elements in graph embedding}
\langle \varphi, Z_{w_1, w_2}(s)\rangle = Z_{w_1, w_2}(\varphi; s) \colonequals \int_X \sigma(w_1)\overline{\sigma(w_2)}\varphi=\sum_{i, j}s^{i + j}\int_X |f|^{2s}m_i\cdot \overline{m_j'}\cdot \varphi.
\end{equation}
Here we use the convention that the complex conjugate of an element of $(\cO_X)_f[s]f^s$ conjugates the coefficients and $f$, and leaves $s$ alone. Note that in the special case $w_1 = w_2 = 1\otimes 1$, we recover $Z_f$ from the introduction as the meromorphic extension of 
\begin{equation}\label{eq: original zeta}
    Z_f(\varphi; s) := Z_{1\otimes 1, 1 \otimes 1}(\varphi; s) = \int_X |f|^{2s} \varphi,
\end{equation}
by identifying test forms with test functions using the standard volume form on $\mb{C}^n$.

Since under the isomorphism \eqref{eqn: malgrange isomorphism} $t$ acts as $s\mapsto s+1$, one can directly check that
\begin{equation}\label{eqn: t is compatible with zeta function}
Z_{t\cdot w_1,t\cdot w_2}(s)=Z_{w_1,w_2}(s+1).
\end{equation}
Arguing as in \cite[\S 12.5.b]{MHMproject} (see also the proof of \eqref{eqn: mero ext of Zmm}), the existence of $b$-functions implies that $Z_{w_1, w_2}$ extends meromorphically to $s \in \mb{C}$ and, furthermore, if $w_1\in V^{\alpha}\iota_{+}(\cO_X)_f$ and $w_2 \in \iota_+ (\cO_X)_f$, then
\begin{equation}\label{eqn: bound of pole order}
\textrm{the poles of $Z_{w_1,w_2}$ are $\leq -\alpha$}, \quad \mbox{the pole order at } s=-\alpha \mbox{ is } \leq   \mult_{s=-\alpha}b_{w_1}(s).
\end{equation}
The idea to consider such zeta functions was first suggested to the third named author by Musta\c{t}\u{a}--Popa \cite{Mustataprivate}.

The zeta functions $Z_{w_1, w_2}$ are related to the polarizations on nearby cycles for the Hodge module $\mc{M} = \mc{O}_X$, equipped with the trivial variation of Hodge structure (i.e., $F_0\mc{O}_X = \mc{O}_X$, $F_{-1}\mc{O}_X = 0$) and trivial polarization $S$ given by
\begin{equation}\label{eqn: trivial polarization}
\langle \varphi, S(g, h) \rangle = \int_X g \bar{h} \varphi.\end{equation}
In this case, we have the following. 

\begin{lemma}\label{lemma: local Archimedean zeta function and Mellin transform}
    For any $w_1,w_2\in \iota_{+}\cO_X$, one has $F_{w_1,w_2}(\varphi;s) = Z_{w_1,w_2}(\varphi;s)$.
\end{lemma}

\begin{proof}
The result holds by definition if $w_1, w_2 \in \iota_*(\mc{O}_X \otimes 1) \subseteq \iota_+\mc{O}_X$. Moreover, writing $s = -\partial_t t$, we have by $\sD_{X \times \mb{C}}$-linearity of $\iota_+S$ that
\begin{align*}
F_{sw_1, w_2}(\varphi; s) &= \langle |t|^{2s} \eta(t - f) \varphi \wedge d\mathrm{Vol}_\mb{C}, \iota_+S(-\partial_t t w_1, w_2) \rangle \\
&= \langle |t|^{2s} \eta(t - f) \varphi \wedge d\mathrm{Vol}_\mb{C}\cdot (-\partial_t t), \iota_+S(w_1, w_2) \rangle \\
&= \langle (t\partial_t |t|^{2s} \eta(t - f)) \varphi \wedge d\mathrm{Vol}_\mb{C},\iota_+S(w_1, w_2) \rangle \\
&= s F_{w_1, w_2}(\varphi; s),
\end{align*}
where we have used that $\eta(t - f)$ has vanishing derivatives in a neighborhood of the support of $\iota_+S(w_1, w_2)$. Similarly, $F_{w_1, sw_2}(\varphi; s) = sF_{w_1, w_2}(\varphi; s)$. Since the same holds for $Z_{w_1, w_2}$, we have $F_{w_1, w_2}(\varphi; s) = Z_{w_1,w_2}(\varphi;s)$ for $w_1, w_2 \in \mb{C}[s] \cdot \iota_*(\mc{O}_X \otimes 1)$. For general $w_1, w_2 \in \iota_+\mc{O}_X$, pick $k \geq 0$ such that $t^k w_1, t^k w_2 \in \mb{C}[s] \cdot i_*(\mc{O}_X \otimes 1)$, so by \eqref{eqn: t is compatible with zeta function} we conclude
\[ F_{w_1, w_2}(\varphi; s) = F_{t^k w_1, t^k w_2}(\varphi; s - k) = Z_{t^k w_1, t^k w_2}(\varphi; s - k) = Z_{w_1, w_2}(\varphi; s).\qedhere\]
\end{proof}

\section{Proof of main results}\label{sec: proof of main results}

In this section, we apply the results from the previous sections to give the proofs of our main results. In \S\ref{subsec:hodge}, we prove a general result (Theorem \ref{thm: sheaf theoretic loeser}) relating Hodge filtrations to poles of zeta functions and use it to prove Theorem \ref{thm: producing poles}. In \S\ref{subsec:minimal exponent}, we prove Theorem \ref{thm: Loesers conjecture}. In \S\ref{subsec:V-filtration theorem}, we prove Theorem \ref{thm: analytic charaterization of Vfiltration} characterizing the $V$-filtration via the zeta function. In \S\ref{subsec:hodge ideals} we prove Theorems \ref{thm:intro Hodge ideals} and \ref{thm:intro weighted ideals} on the (weighted) higher multiplier ideals. Finally, in \S\ref{subsec:loeser counterexample}, we give the proof of Theorem \ref{thm: counterexample to Loeser}, which shows that not every root of $b_f(s)$ need be a pole of $Z_f$, as well as Corollary \ref{cor: Budur Walther question for minimal exponent}.

\subsection{Multiplicities, pole orders and the Hodge filtration} \label{subsec:hodge}

We begin with the following general result linking zeros of $b$-functions to poles of Archimedean zeta functions.  As before, let $X$ be a complex manifold and $f : X \to \mb{C}$ a holomorphic function. For $w \in \iota_+(\mc{O}_X)_f$, let us write the zeta function from \eqref{eqn:zeta function for elements in graph embedding} as $Z_w\colonequals Z_{w,w}$. For $\alpha \in \mb{Q}$ such that $b_w(-\alpha) = 0$, we define
\begin{equation}\label{eqn: center of alpha}C(w,\alpha) \colonequals \{ x \in X \mid \mult_{s = -\alpha} b_{w, x}(s) = \mult_{s = -\alpha} b_w(s)\} \subseteq f^{-1}(0),\end{equation}
where $b_{w, x}(s)$ is the local $b$-function of $w$ near $x$. 

\begin{thm}\label{thm: sheaf theoretic loeser}
Let $\alpha\geq 0$ and let $k$ be the smallest integer such that $F_{k+1}\gr^{\alpha}_V\iota_{+}\cO_X\neq 0$. Let $w \in F_{k + 1} V^{\alpha}\iota_+\cO_X$ and $\varphi$ be a local volume form such that $0 \neq [w] \in \gr_V^\alpha \iota_+\cO_X$ and the support of $\varphi$ intersects $C(w,\alpha)$. Then the following hold.
   \begin{enumerate}
       \item \label{itm:sheaf theoretic loeser 1}If $\alpha>0$, then $s=-\alpha$ is a pole of $Z_{w}(\varphi;s)$ of order equal to $\mult_{s=-\alpha}b_w(s)$.
       \item \label{itm:sheaf theoretic loeser 2}If $\alpha=0$ and $w'\in V^{1}\iota_{+}\cO_X$ satisfies $\d_t[w']=[w]\in \gr^0_V\iota_{+}\cO_X$, then $s=-1$ is a pole of $Z_{w'}(\varphi;s)$ of order equal to $\mult_{s=0}b_w(s)+1$.
   \end{enumerate}
\end{thm}
\begin{proof}
Observe that if $\alpha > 1$ then $t : V^{\alpha - 1}\iota_+\mc{O}_X \to V^\alpha \iota_+\mc{O}_X$ is an isomorphism, $b_{tw}(s) = b_w(s + 1)$, and $Z_{tw}(s) = Z_w(s + 1)$. So we may assume without loss of generality that $\alpha \in [0, 1]$. In this case, $(\gr_V^\alpha\iota_+\mc{O}_X, F_\bullet, \op{N}, \gr_V^\alpha(S))$ is a polarized Hodge-Lefschetz module. By Lemma \ref{lemma: lowest hodge is primitive} and Corollary \ref{cor:Nmu}, the section $[w]$ has monodromy weight $\ell \colonequals \mult_{s = -\alpha} b_w(s) - 1$ and defines a primitive element
\[ [[w]] \in \op{P}\gr_\ell^{W(\op{N})}\gr_V^\alpha\iota_+\mc{O}_X, \]
whose support is equal to $C(w,\alpha)$.

To prove \eqref{itm:sheaf theoretic loeser 1}, observe that by Corollary \ref{cor:HL lowest piece is positive}, \eqref{eqn: definition of polarization on grValpha} and Lemma \ref{lemma: local Archimedean zeta function and Mellin transform}, we have
\begin{align*}
0\neq \langle \varphi, \gr^{\alpha}_V(S)([w],\op{N}^{\ell}[w])\rangle\colonequals\Res_{s=-\alpha}(s+\alpha)^{\ell}F_{w,w}(\varphi;s)=\Res_{s=-\alpha}(s+\alpha)^{\ell}Z_{w}(\varphi;s),
\end{align*}
so $Z_w(\varphi; s)$ has a pole of order at least $\ell + 1$ at $s = -\alpha$. Conversely by \eqref{eqn: selfadjointness}
\[ \langle \varphi, \gr^{\alpha}_V(S)([w],\op{N}^{\ell+a}[w])\rangle=0, \quad \text{for $a > 0$},\]
since $\op{N}^{\ell + a}[[w]] = 0 \in \gr_{-\ell - 2a}^{W(\op{N})}\gr_V^\alpha\iota_+\mc{O}_X$, so the order of the pole is exactly $\ell + 1$ as claimed.

To prove \eqref{itm:sheaf theoretic loeser 2}, applying Proposition \ref{prop: relation between polarization grVo and grV1} and recalling that $\op{can} = -\d_t$, we have
\begin{align*}
\Res_{s = -1} (s + 1)^{\ell + 1} Z_{w'}(\varphi; s) &= \langle \varphi, \gr_V^1(S)([w'], \op{N}^{\ell + 1}[w'])\rangle \\
&= \langle \varphi, \gr_V^1(S)([w'], \op{var} \circ \op{N}^\ell \circ \op{can} [w'])\rangle \\
&= - \langle \varphi, \gr_V^0(S)(\op{can} [w'], \op{N}^\ell \circ \op{can} [w']) \rangle \\
&= - \langle \varphi, \gr_V^0(S)([w], \op{N}^\ell [w])\rangle \neq 0.
\end{align*}
Similarly, $\Res_{s=-1}(s+1)^{\ell+1+a}Z_{w'}(\varphi;s)=0$ for any integer $a\geq 1$, so $s=-1$ is a pole of $Z_{w'}(\varphi;s)$ of order $\ell + 2 = \mult_{s = 0}b_w(s) + 1$ as claimed.
\end{proof}

Theorem \ref{thm: sheaf theoretic loeser} implies Theorem \ref{thm: producing poles} from the introduction.

\begin{proof}[Proof of Theorem \ref{thm: producing poles}]
As in the statement of the theorem, let $\alpha \in [0, 1)$ satisfy $\gr_V^\alpha\iota_+\mc{O}_X \neq 0$ and let $k$ be the smallest integer such that $F_{k + 1} \gr_V^\alpha \iota_+\mc{O}_X \neq 0$.

If $\alpha> 0$, choose $w\in F_{k+1}V^{\alpha}\iota_{+}\cO_X$ such that $0\neq [w]\in F_{k+1}\gr^{\alpha}_V\iota_{+}\cO_X$ and write $w=\sum_{0\leq i\leq k} g_i \otimes \d_t^i$ with $g_i \in \cO_X$. Under the isomorphism \eqref{eqn: malgrange isomorphism}, $w$ is sent to
\[ \sum_{0\leq i\leq k} s^i\cdot \frac{h_i}{f^k}\cdot f^s\in (\cO_X)_f[s]f^s, \]
where $h_i\in \cO_X$. Hence if $\varphi$ is a local volume form such that the support of $\varphi$ intersects $C(w,\alpha)$, then
\[ Z_{w}(\varphi;s)=\sum_{0\leq i,j\leq k} s^{i+j}Z_{f}(h_i\overline{h_j}\varphi;s-k).\]
By Theorem \ref{thm: sheaf theoretic loeser}, $-\alpha$ is a pole of $Z_{w}(\varphi;s)$ of order $\mult_{s=-\alpha}b_w(s)$, so there must exist $i,j$ such that $-\alpha$ is a pole of $Z_{f}(h_i\overline{h_j}\varphi;s-k)$ of order at least $\mult_{s=-\alpha}b_w(s)$. In particular, $-\alpha-k$ is a pole of $Z_f$ of order at least 
\[\max\{\mult_{s=-\alpha}b_w(s) \mid w\in F_{k+1}V^{\alpha}\iota_{+}\cO_X\}=\min\{r \mid \op{N}^r\cdot F_{k+1}\gr^{\alpha}_V\iota_{+}\cO_X=0\},\]
by Lemma \ref{lemma: lowest hodge is primitive} and Corollary \ref{cor:Nmu}. As $\alpha\neq 0$, we obtain the same conclusion for $\tilde{Z}_f$.

Similarly, if $\alpha=0$, since $\mc{O}_X$ is a Hodge module with strict support $X$, the morphism
\[ \op{can} = -\partial_t : (\gr_V^1\iota_+\mc{O}_X, F_\bullet) \to (\gr_V^0\iota_+\mc{O}_X, F_{\bullet + 1})\]
is strict surjective. So we may choose $w' \in F_kV^1\iota_+\cO_X$ such that $\partial_t [w'] \in F_{k + 1}\gr_V^0\iota_+\mc{O}_X$ is non-zero. Applying the argument above to $w'$, we conclude that $-k$ is a pole of $Z_f$ of order at least
\[ \max\{ \mult_{s=0}b_w(s)+1 \mid w\in F_{k+1}\gr^0_V\iota_{+}\cO_X\}=\min\{r\mid N^r\cdot F_{k+1}\gr^0_V\iota_{+}\cO_X=0\}+1\geq 2, \]
and hence a pole of $\tilde{Z}_f$ with the desired bound on order.

Finally, suppose that $\alpha + k < \tilde{\alpha}_f + 1$. Since $\tilde{Z}_f$ has a pole at $ - \alpha - k$, we must have $\tilde{b}_f(-\alpha - k + m) = 0$ for some integer $m \geq 0$. Since $-\tilde{\alpha}_f$ is the largest root of $\tilde{b}_f(s)$, we must have $m = 0$, so $-\alpha - k$ is a root as claimed.
\end{proof}

Theorem \ref{thm: producing poles} implies the following criterion for poles of $\tilde{Z}_f$ in terms of more classical Hodge theory. Denote by $M_x$ the Milnor fiber of $f$ at $x\in X$. The reduced cohomology group $\tilde{H}^p(M_x,\Q)$ underlies a (rational) mixed Hodge structure which satisfies \cite[Proposition 5.2, (5.3.1)]{BS05}:
\[ \gr_F^j\tilde{H}^p(M_x,\mb{C})=0, \quad j<0, \textrm{ or } j>p.\]
Denote by $\tilde{H}^p(M_x,\mb{C})_{e^{-2\pi i\alpha}}$ the eigenspace of the local monodromy, with $\alpha \in [0,1)$. 
\begin{prop}\label{prop: strengthen Barlet}
Suppose $\tilde{H}^p(M_x,\mb{C})_{e^{-2\pi i\alpha}}\neq 0$ for some $\alpha\in [0,1)$. Let $j\in [0,p]$ be an integer such that $F^j\tilde{H}^p(M_x,\mb{C})_{e^{-2\pi i\alpha}}\neq 0$, then $\tilde{Z}_f$ has a pole at 
\[ \begin{cases}
    -\alpha-(p-j), \quad &\textrm{if $\alpha\in (0,1)$},\\
    -\alpha-(p-j)-1, \quad &\textrm{if $\alpha=0$}.
\end{cases}\]
\end{prop}
\begin{proof}
It follows from \cite[(4.1.1),(5.1.1)]{BS05} that 
\begin{equation}\label{eqn: cohomology of Milnor fiber} 
 (\tilde{H}^p(M_x,\Q)_{e^{-2\pi i\alpha}},F^{\bullet})=(\cH^{p-n+1}i^{\ast}_x\phi_{f,e^{-2\pi i\alpha}}\mc{O}_X,F_{-\bullet}), \quad \text{for $\alpha \in [0,1)$,}
\end{equation}
where we write $\phi_{f, e^{-2\pi i \alpha}} = \psi_{f, e^{-2\pi i\alpha}}$ if $\alpha \in (0, 1)$. Thus, by Lemma \ref{lem:hodge restriction} below, we have
\[ F_{-j + p + 1}\phi_{f, e^{-2\pi i\alpha}}\mc{O}_X \neq 0.\]
Since by \eqref{eqn: shifts of Hodge} we have
\[ F_\bullet \phi_{f, e^{-2\pi i \alpha}}\mc{O}_X = \begin{cases} F_\bullet \gr_V^\alpha \iota_+\mc{O}_X, & \text{if $\alpha \in (0, 1)$}, \\ F_{\bullet + 1}\gr_V^0\iota_+\mc{O}_X, & \text{if $\alpha = 0$},\end{cases}\]
we deduce the proposition using Theorem \ref{thm: producing poles} and the fact that if $\beta$ is a pole of $\tilde{Z}_f$ then so is $\beta - 1$. 
\end{proof}

\begin{lemma} \label{lem:hodge restriction}
Let $\mc{M}$ be a mixed Hodge module on $X$ and $i \colon Z \to X$ the inclusion of a closed complex submanifold of codimension $r$. If $F_k \mc{M} = 0$ then $F_{k - p - r}\mc{H}^p i^*\mc{M} = 0$. 
\end{lemma}
\begin{proof}
By induction, it suffices to prove the lemma when $r = 1$. In this case, we have
\[ i_*i^*\mc{M} = [\psi_{g, 1}\mc{M} \xrightarrow{\op{can}} \phi_{g, 1}\mc{M}],\]
placed in degrees $-1$ and $0$, where $g$ is a local equation for $Z$. The lemma follows since $F_k \phi_{g, 1}\mc{M} = F_{k + 1}\psi_{g, 1}\mc{M} = 0$ and $i_*$ increases the lowest piece of $F_\bullet$ by $1$.
\end{proof}

\begin{remark}\label{remark: sharpen barlet}
Proposition \ref{prop: strengthen Barlet} gives a Hodge-theoretic improvement of \cite[Th\'eor\`eme 2]{Barlet86}, which shows that if $H^p(M_x,\Q)_{e^{-2\pi i\alpha}}\neq 0$ with $p \geq 1$ (in the case $\alpha = 0$ further requiring $\op{N} \neq 0$ on $H^p(M_x, \Q)_{1}$), then $-\alpha-p$ is a pole of $\tilde{Z}_f$. Proposition \ref{prop: strengthen Barlet} clearly implies this when $\alpha \neq 0$. When $\alpha = 0$, since  $\op{N}$ is of type $(-1, -1)$, the requirement $\op{N} \neq 0$ on $\tilde{H}^p(M_x, \Q)_{1}=H^p(M_x,\Q)_{1}$ implies that $\tilde{H}^p(M_x, \Q)_{1}$ has more than one non-zero Hodge number, and hence that we can take $j \geq 1$ in Proposition \ref{prop: strengthen Barlet}.

Barlet's result also gives a bound on the pole order in terms of the order of the monodromy operator $\op{N}$; using the techniques in the proof of Theorem \ref{thm: sheaf theoretic loeser}, it is possible to derive analogous bounds in the setup of Proposition \ref{prop: strengthen Barlet}.

\end{remark}

\subsection{The minimal exponent} \label{subsec:minimal exponent}

We now turn to Theorem \ref{thm: Loesers conjecture}. We begin with the following lemma.

\begin{lemma}\label{lem:mink}
    Write $\tilde{\alpha}_f=k+\alpha$ for a (unique) $k\in \mb{Z}_{\geq 0}$ and $\alpha\in [0,1)$. Then,
    \[ 0\neq [1\otimes \d_t^k]\in F_{k+1}\gr^{\alpha}_V\iota_{+}\cO_X, \quad F_{k}\gr^{\alpha}_V\iota_{+}\cO_X=0.\]
\end{lemma}

\begin{proof}
By \cite[Lemma 3.10]{SY23} (see also \cite{Saito16}), we have
    \begin{equation}\label{eqn: SchnellYang description of minimal exponent} \tilde{\alpha}_f = \min \{ j+\beta, j\in \mb{Z}_{\geq 0},\beta \in [0,1) \mid \gr^F_{j+1}\gr^{\beta}_V\iota_{+}\cO_X\neq 0\}.\end{equation}
It follows that $F_{k}\gr^{\alpha}_V\iota_{+}\cO_X=0$. Clearly, $1\otimes \d_t^k \in F_{k+1}\iota_+ \cO_X$, and by Theorem \ref{thm: Sabbah} and Proposition \ref{prop:sharp}, we have $1\otimes \d_t^{k}\in V^{\alpha}\iota_{+}\cO_X\setminus V^{>\alpha}\iota_{+}\cO_X$.
\end{proof}

\begin{proof}[Proof of Theorem \ref{thm: Loesers conjecture}]
Write $\tilde{\alpha}_f = k + \alpha$ for $k \in \mb{Z}_{\geq 0}$ and $\alpha \in [0, 1)$ as in Lemma \ref{lem:mink} and set $w = 1 \otimes \partial_t^k \in F_{k + 1} V^\alpha \iota_+\mc{O}_X$. By Proposition \ref{prop:sharp}, we have $b_w(s) = (s + 1)\tilde{b}_f(s - k)$. Since this holds locally near each point, $C(w,\alpha)$ from \eqref{eqn: center of alpha} is therefore equal to 
\begin{equation}\label{eqn: loci reaching maximal pole order} \{ x \in X \mid \mult_{s = -\tilde{\alpha}_f} b_{f, x}(s) = \mult_{s = -\tilde{\alpha}_f} b_f(s)\}.\end{equation}
Fix a local volume form $\varphi$ whose support intersects $C(w,\alpha)$.

Suppose that $\tilde{\alpha}_f \not\in\mb{Z}$ so $\alpha > 0$. Then by Theorem \ref{thm: sheaf theoretic loeser}, we have that $ -\alpha$ is a pole of
\[ Z_w(\varphi; s) = \prod_{i = 0}^{k - 1} (-s + i)^2 \cdot Z_f(\varphi; s - k) \]
of order $\mult_{s = -\alpha} b_w(s) = \mult_{s = -\tilde{\alpha}_f} \tilde{b}_f(s)$, so $ -\tilde{\alpha}_f$ is a pole of $\tilde{Z}_f(\varphi; s)$ of the same order.

If, on the other hand, $\tilde{\alpha}_f \in \mb{Z}$, then Proposition \ref{prop:sharp} and Theorem \ref{thm: Sabbah} imply that $w' \colonequals 1 \otimes \partial_t^{k - 1} \in V^1\iota_+\mc{O}_X$. Applying Theorem \ref{thm: sheaf theoretic loeser} to $w'$, we see that $s=-1$ is a pole of
\[ Z_{w'}(\varphi;s)=\prod_{i=0}^{k-2}(-s+i)^2\cdot Z_{f}(\varphi;s-(k-1))\]
with order $\mult_{s = 0}b_w(s) + 1 = \mult_{s = -\tilde{\alpha}_f} \tilde{b}_f(s) + 1$. Hence $s=-\tilde{\alpha}_f = -k$ is a pole of $\tilde{Z}_f(\varphi;s)$ of order $\mult_{s = -\tilde{\alpha}_f} \tilde{b}_f(s)$ as claimed.
\end{proof}

\begin{remark} It follows from the above proof 
that for any local volume form $\varphi$, the pole order of $\tilde{Z}_f(\varphi;s)$ at $s=-\tilde{\alpha}_f$ equals $ \max_{x\in \mathrm{supp} \, \varphi}\{ \mult_{s=-\tilde{\alpha}_f}\tilde{b}_{f,x}(s)\}$.
\end{remark}

We connect the locus \eqref{eqn: loci reaching maximal pole order} to the center of minimal exponent $(X,\mathrm{div}(f))$ in \cite[Definition 4.2]{SY23}, which is defined to be the support of $F_{k+1}\gr^{W(\op{N})}_{\ell}\gr^{\alpha}_V\iota_{+}\cO_X$, where $\ell$ is the highest such that  $F_{k+1}\gr^{W(\op{N})}_{\ell}\gr^{\alpha}_V\iota_{+}\cO_X\neq 0$. 
\begin{corollary}\label{corollary: determine the center of minimal exponent}
One has $[[1\otimes \d_t^k]]$ generates $F_{k+1}\gr^{W(\op{N})}_{\ell}\gr^{\alpha}_V\iota_{+}\cO_X$ as an $\cO$-module and
\[ \ell+1=\min \{ r \mid (s+\alpha)^{r}\cdot F_{k+1}\gr^{\alpha}_V\iota_{+}\cO_X=0\}=\mathrm{mult}_{s=-\tilde{\alpha}_f}\tilde{b}_f(s).\]
In particular, the center of minimal exponent agrees with the locus \eqref{eqn: loci reaching maximal pole order}.
\end{corollary}
\begin{proof}
As $1\otimes \d_t^i\in V^{>\alpha}\iota_{+}\cO_X$ for $i\leq k-1$, $[1\otimes \d_t^k]$ generates $F_{k+1}\gr^{\alpha}_V\iota_{+}\cO_X$. Hence,
    \begin{align*}
         \min \{ r \mid (s+\alpha)^{r}\cdot F_{k+1}\gr^{\alpha}_V\iota_{+}\cO_X=0\}=\min \{ r \mid (s+\alpha)^{r}\cdot [1\otimes \d_t^k]=0\in \gr^{\alpha}_V\iota_{+}\cO_X\}.
 \end{align*}
The rest follows from Lemma \ref{lem:mink} and \eqref{eqn: weight level}.
\end{proof}

\subsection{Analytic characterization of the $V$-filtration} \label{subsec:V-filtration theorem}

\begin{proof}[Proof of Theorem \ref{thm: analytic charaterization of Vfiltration}]
Let us write
\[ U^\alpha = \{ w \in \iota_+(\mc{O}_X)_f \mid \text{for all $w' \in \iota_+ (\cO_X)_f$, $Z_{w,w'}$ has no poles $> -\alpha$}\}.\]
It follows from Theorem \ref{thm: Sabbah} and \eqref{eqn: bound of pole order} that $V^\alpha \iota_+(\mc{O}_X)_f \subseteq U^\alpha$. Conversely, suppose that $w \in \iota_+(\mc{O}_X)_f$ is not in $V^\alpha \iota_+(\mc{O}_X)_f$. Then we have $w \in V^\beta \iota_+(\mc{O}_X)_f \setminus V^{>\beta} \iota_+(\mc{O}_X)_f$ for some $\beta <  \alpha$. Since $t$ acts bijectively on $\iota_+(\mc{O}_X)_f$, we may write
\[ t^{-\lceil \beta \rceil + 1} w \in V^{\gamma} \iota_+(\mc{O}_X)_f \setminus V^{>\gamma} \iota_+(\mc{O}_X)_f,\]
where $\gamma = \beta - \lceil \beta \rceil + 1 \in (0, 1]$. Now, by \cite[\S 3.e]{Sabbah02}, the pairing $\gr_V^{\gamma}(S)$ on $\gr_V^{\gamma} \iota_+(\mc{O}_X)_f$ is non-degenerate; this also follows from the fact that these form part of a polarized Hodge--Lefschetz module, but it is much more elementary. Hence, by Lemma \ref{lemma: local Archimedean zeta function and Mellin transform} and \eqref{eqn: t is compatible with zeta function}, there exists $w' \in \iota_+(\mc{O}_X)_f$ such that $t^{-\lceil \beta \rceil + 1} w'\in V^{\gamma}\iota_+(\mc{O}_X)_f$ and
\[ \Res_{s = -\beta} Z_{w, w'}(s) = \Res_{s = -\gamma} Z_{t^{-\lceil \beta \rceil + 1} w, t^{-\lceil \beta \rceil + 1 } w'}(s) = \gr_V^{\gamma}(S)(t^{-\lceil \beta \rceil + 1} w, t^{-\lceil \beta \rceil + 1} w') \neq 0.\]
So $Z_{w, w'}(s)$ has a pole at $-\beta > -\alpha$, so $w \not\in U^\alpha$.

Finally, suppose that $w \in V^\alpha\iota_+(\mc{O}_X)_f$. Clearly the pole order of $Z_{w, w'}$ at $s = -\alpha$ is at most $\ell \colonequals \mult_{s = -\alpha} b_w(s)$ by \eqref{eqn: bound of pole order}. If $\ell=0$, there is nothing to prove. Assume $\ell\geq 1$. Conversely, by Corollary \ref{cor:Nmu}, we have $0 \neq (s + \alpha)^{\ell - 1} [w] \in \gr_V^\alpha \iota_+(\mc{O}_X)_f$, so there exists $w'$ such that
\[ Z_{(s + \alpha)^{\ell - 1}w, w'}(s) = (s + \alpha)^{\ell - 1} Z_{w, w'}(s) \]
has a pole at $s = -\alpha$, so $Z_{w, w'}(s)$ has a pole of order $\ell$ as claimed.
\end{proof}

\begin{remark}\label{remark: Prop 1.5 is false for OX}
In contrast, we claim that there exist $f,\alpha$ such that
    \begin{equation}\label{eqn: Prop 1.5 false for OX}
    V^{\alpha}\iota_{+}\cO_X\subsetneq\{w\in \iota_{+}\cO_X \mid \forall w' \in \iota_+ \cO_X, \textrm{all poles of $Z_{w,w'}$ are $\leq -\alpha$}\}.\end{equation}
Choose $f$ so that $\gr_V^0\iota_+\mc{O}_X \neq 0$ (e.g. $f=x^2+y^2$). Choose $w \in V^0 \iota_+\mc{O}_X \setminus V^{>0}\iota_+\mc{O}_X$ satisfying $(\partial_t t)\cdot w \in V^{>0}\iota_+\mc{O}_X$. Then for any $w' \in \iota_+\mc{O}_X$, we can write $w' \in \partial_t w'' + V^{>0}\iota_+\mc{O}_X$ for some $w'' \in \iota_+\mc{O}_X$. So by \eqref{eqn: bound of pole order}, $Z_{w, w'}(s)$ has a pole at $s = 0$ if and only if $Z_{w, \partial_t w''}(s)$ has one. But $Z_{w, \partial_t w''}(s) = -s Z_{tw, w''}(s - 1) = Z_{t \cdot \partial_t t w, w''}(s - 1)$, which has no pole at $s = 0$ since $t \cdot \partial_t t w \in V^{>1}\iota_+\mc{O}_X$. Hence, all poles of $Z_{w, w'}$ are $\leq -\alpha$ for some $\alpha > 0$, but $w \not\in V^\alpha\iota_+\mc{O}_X$.

Furthermore, it follows from \eqref{eqn: Prop 1.5 false for OX} that there exist $f$ and $\alpha$ such that
\[V^{\alpha}\iota_{+}(\cO_X)_f\subsetneq \{w\in \iota_{+}(\cO_X)_f \mid \textrm{all poles of $Z_{w,w}$ are $\leq -\alpha$}\}.\]
\end{remark}

\begin{remark}\label{remark:analytic description of Hodge ideals and higher multiplier ideals}
Theorem \ref{thm: analytic charaterization of Vfiltration} gives, in principle, an analytic description of the Hodge ideals using \cite[Theorem A]{MPVfiltration} (see also \cite[(1.7)]{DY25}): 
\begin{equation}
I_k(f^\alpha)\otimes \cO_X(kD)\cdot f^{-\alpha}= \mathrm{ev}_{s=-\alpha}\circ \sigma\left(F_{k+1}\iota_{+}\cO_X\cap V^{\alpha}\iota_{+}(\cO_X)_f\right), \quad \alpha >0,
\end{equation}
where $D=\mathrm{div}(f)$, $\sigma$ comes from \eqref{eqn: malgrange isomorphism} and $\mathrm{ev}_{s=-\alpha}$ denotes the evaluation map 
\[(\cO_X)_f[s]f^s\xrightarrow{s=-\alpha} (\cO_X)_f\cdot f^{-\alpha}.\] Similarly, we obtain an analytic description of the higher multiplier ideals via 
\[ \cJ_{k}(f^{\alpha})\otimes \cO_X(kD)\colonequals \frac{F_{k+1}\iota_{+}\cO_X\cap V^{>\alpha}\iota_{+}(\cO_X)_f}{F_{k}\iota_{+}\cO_X\cap V^{>\alpha}\iota_{+}(\cO_X)_f}, \quad \alpha\geq 0.\]
\end{remark}

\subsection{Hodge and higher multiplier ideals} \label{subsec:hodge ideals}

Finally, let us give the proofs of Theorems \ref{thm:intro Hodge ideals} and \ref{thm:intro weighted ideals}. Recall from \cite[\S 1.6]{SY23} that the higher multiplier ideals $\mc{J}_k(f^\alpha)$ (which is $\cI_{k,<-\alpha}(\textrm{div}(f))$ in \cite[Definition 3.2]{SY23}) are defined by
\[ \mc{J}_k(f^\alpha) \otimes \partial_t^k = \gr^F_{k + 1} V^{> \alpha} \iota_+\mc{O}_X \subseteq \gr^F_{k + 1} \iota_+\mc{O}_ X= \mc{O}_X \otimes \partial_t^k.\]
The weighted version $W_\ell \mc{J}_k(f^{\alpha - \epsilon}) \subseteq \mc{J}_k(f^{\alpha - \epsilon})$ is defined as the pre-image of
\[\gr^F_{k + 1} W(\op{N})_{\ell-1} \gr_V^{\alpha}\iota_+\mc{O}_X\]
under the natural quotient map. 

\begin{proof}[Proof of Theorem \ref{thm:intro Hodge ideals}]
Let us write
\[ \mc{J}_k'(f^\alpha) = \{g \in \mc{O}_X \mid \text{for all $\varphi$, $\tilde{Z}_f(|g|^2\varphi; s)$ has no poles $\geq -\alpha - k$}\};\]
we need to show that $\mc{J}_k'(f^\alpha) = \mc{J}_k(f^\alpha)$ for $\alpha + k  < \tilde{\alpha}_f + 1$.

If $g \in \mc{J}_k(f^\alpha)$, then $g \otimes \partial_t^k \in \gr^F_{k + 1} V^{>\alpha}\iota_+\mc{O}_X$. Since $\alpha + k < \tilde{\alpha}_f + 1$, we have by \eqref{eqn: SchnellYang description of minimal exponent} that $\mc{J}_{i}(f^\alpha) = \mc{O}_X$ and hence $\gr^F_{i + 1} V^{>\alpha}\iota_+\mc{O}_X = \mc{O}_X \otimes \partial_t^i$ for all $i < k$. This implies that $g \otimes \partial_t^k \in V^{>\alpha}\iota_+\mc{O}_X$, so by Theorem \ref{thm: analytic charaterization of Vfiltration}, we have that
\[ Z_{g \otimes \partial_t^k, g/f^k \otimes 1}(\varphi; s) = Z_f(|g|^2\varphi; s - k)\cdot \prod_{j = 0}^{k - 1}(-s+j)  \]
has no poles $\geq -\alpha$. Hence $g \in \mc{J}_k'(f^\alpha)$.

Conversely, suppose that $g \not\in \mc{J}_k(f^\alpha)$. By \cite[Proposition 3.7]{SY23}\footnote{One can also use \cite[Corollary 3.8]{SY23} and the fact that the microlocal $V$-filtration on $\cO_X$ is decreasing \cite[\S 1.3]{Saito16}.}, the ideals $\mc{J}(\gamma) \colonequals \mc{J}_{\lfloor \gamma \rfloor}(f^{\gamma - \lfloor \gamma \rfloor - \epsilon})$ satisfy $\mc{J}(0) = \mc{O}_X$ and $\mc{J}(\gamma) \subseteq \mc{J}(\gamma')$ for $\gamma > \gamma'$. So there exists $j \geq 0$ and $\beta \in [0, 1)$ with $j + \beta \leq k + \alpha$ such that $g \in \mc{J}(\beta + j) \setminus \mc{J}(\beta + j + \epsilon)$, i.e., 
\[g \otimes \partial_t^j \in \gr^F_{j + 1} V^{\beta}\iota_+\mc{O}_X \setminus \gr^F_{j + 1} V^{>\beta} \iota_+\mc{O}_X.\]
Since $\beta + j < \tilde{\alpha}_f + 1$, this implies that $w \colonequals g \otimes \partial_t^j \in V^\beta \iota_+\mc{O}_X$ as above, that $0 \neq [w] \in F_{j + 1}\gr_V^\beta \iota_+\mc{O}_X$ and that $F_j \gr_V^\beta \iota_+\mc{O}_X = 0$. Applying the same argument as the proof of Theorem \ref{thm: Loesers conjecture}, we deduce that $\tilde{Z}_{f}(|g|^2\varphi; s)$ has a pole of order
\[ \max_{x \in \supp \varphi} \mult_{s = -\beta} b_{w, x}(s)\]
at $s = - \beta - j \geq -\alpha - k$ for any local volume form $\varphi$ whose support intersects $C(w,\beta)$. So $g \not\in \mc{J}_k'(f^\alpha)$ and we are done.
\end{proof}

\begin{proof}[Proof of Theorem \ref{thm:intro weighted ideals}]
Let $g \in \mc{J}_k(f^{\alpha - \epsilon})$; we need to show that the maximal pole order of $\tilde{Z}_f(|g|^2\varphi; s)$ at $s = -\alpha - k$ as $\varphi$ varies is one more than the monodromy weight of $0\neq [w] \in F_{k + 1} \gr_V^\alpha\iota_+\mc{O}_X$, where $w = g \otimes \partial_t^k$. Since the maximal pole order can be detected by some local volume form $\varphi$, the argument above shows that this is $\mult_{s = -\alpha} b_w(s)$. But this is one more than the monodromy weight by Lemma \ref{lemma: lowest hodge is primitive} and Corollary \ref{cor:Nmu}, so we are done.
\end{proof}

\subsection{A root of $b_f(s)$ which is not a pole of $Z_f$} \label{subsec:loeser counterexample} In this subsection, let $j: U= f^{-1}(\C^\ast) \to X$ denote the open embedding, and  $\cM_f$ the localization along $f$ of a $\sD_X$-module $\cM$. First, we prove a more general result, likely of independent interest.

\begin{prop}\label{prop: V>alpha surjects}
    Let $\cM$ be a simple regular holonomic $\sD_X$-module with quasi-unipotent monodromy along $f$, on which multiplication by $f$ is injective. Then for any $\alpha \in \Q$, the evaluation map $\mathrm{ev}_{s=-\alpha}: \cM_f[s]f^s \to \cM_f \cdot f^{-\alpha}$ restricts to a surjection
    \[\begin{tikzcd}[column sep=large]
    V^{>\alpha} \iota_+ \cM_f \ar[r, twoheadrightarrow, "s\mapsto -\alpha"] & j_{!\ast} (j^*\cM \cdot f^{-\alpha}).
    \end{tikzcd}\]
\end{prop}

\begin{proof}
   Multiplying by a sufficiently high power of $t$, we may assume that $\alpha>0$, so that $V^{>\alpha} \iota_+ \cM_f = V^{>\alpha} \iota_+ \cM$. Let $m f^{-\alpha} \in j_{!\ast} (j^*\cM \cdot f^{-\alpha}) \subseteq \cM_f \cdot f^{-\alpha}$ be a non-zero local section. We can take $\ell \gg 0$ such that $w_0:=t^\ell \cdot (mf^{s}) = m f^{s+\ell} \in V^{>\alpha}\iota_+ \cM$.  Set $\rho := (\mathrm{ev}_{s=-\alpha})_{|V^{>\alpha}\iota_+ \cM}$. Notice that $0 \neq \rho(w_0) = m f^{\ell-\alpha} \in j_{!\ast} (j^*\cM \cdot f^{-\alpha})$, in particular $\rho \neq 0$. 
   
      Take an arbitrary $w\in V^{>\alpha}\iota_+ \cM$. Since $\iota_+ \cM$ is a simple $\sD_{X\times \C}$-module by Kashiwara's theorem, there exists $P(t,\d_t)\in \sD_{X\times \C}$ with $P(t, \d_t) \cdot w_0 = w.$
    By further multiplying with a sufficiently high power $k\in \Z_{\geq 0}$ of $t$, we can eliminate $\partial_t$ and find $P' \in \sD_X\langle s,t\rangle$ with
    \begin{equation}\label{eq:first}    P'(s, t) \cdot w_0 = t^{k} \cdot w.
    \end{equation}
    Since $V^{>\alpha}\iota_{+}\cM$ is a $\sD_X\langle s,t\rangle$-module, using Theorem \ref{thm: Sabbah} repeatedly we can find some $P'' \in \sD_X\langle s,t\rangle$, and $b(s) \in \C[s]$ with $b(-\alpha) = 1$, such that 
    \begin{equation}\label{eq:second}
    P''(s,t) \cdot  (t^k\cdot w)= b(s) \cdot w, 
    \end{equation}        
    Multiplying (\ref{eq:first}) by $P''(s,t)$, using (\ref{eq:second}) and the fact that $t^i \cdot w_0 = f^i \cdot w_0$, we obtain
    \begin{equation}\label{eq:citable}
    Q(s) \cdot w_0 = b(s) \cdot w,
    \end{equation}
    for some $Q(s) \in \sD_X[s]$. Applying $\rho$ to (\ref{eq:citable}), we get $\rho(w)=\rho(b(s)w) = Q(-\alpha) \cdot \rho(w_0) \in j_{!\ast} (j^*\cM \cdot f^{-\alpha})$. Thus, we have $\mathrm{im}\, \rho \subseteq j_{!\ast} (j^*\cM \cdot f^{-\alpha})$. Since $\rho \neq 0$ is a $\sD_X$-module map and $j_{!\ast} (j^*\cM \cdot f^{-\alpha})$ is a simple $\sD_X$-module, we must have $\mathrm{im}\, \rho = j_{!\ast} (j^*\cM \cdot f^{-\alpha})$.
\end{proof}

Theorem \ref{thm: counterexample to Loeser} is now a consequence of the following result.
\begin{prop}\label{prop: an estimate for pole order}
Let $f$ be a holomorphic function on an open subset $X\subseteq \C^n$ with Bernstein-Sato polynomial $b_f(s)$. Let $\alpha\in \Q$ so that $b_f(-\alpha)=0$ and $b_f(-\alpha+i)\neq 0$ for any $i\in \Z_{\geq 1}$. If $f^{-\alpha}\in \Gamma(X,\sD_X\cdot f^{-\alpha+1})$, then $-\alpha$ is a pole of $Z_f$ of order less than $\mult_{s=-\alpha}b_f(s)$.
\end{prop}

\begin{proof}
   Since $b_f(-\alpha+i)\neq 0$ for any $i\in \Z_{\geq 1}$, using the function equation for $b_f(s)$ we obtain $\sD_X\cdot f^{-\alpha+1}= j_{!\ast}(\cO_U\cdot f^{-\alpha})$. We identify $\iota_{+}(\cO_X)_f$ with $(\cO_X)_f[s]f^s$ using the isomorphism \eqref{eqn: malgrange isomorphism}. By Proposition \ref{prop: V>alpha surjects}, the evaluation map $\mathrm{ev}_{s=-\alpha}: \iota_+ (\cO_X)_f \to (\cO_X)_f \cdot f^{-\alpha} $ restricts to a surjection $V^{>\alpha}\iota_{+}(\cO_X)_f\twoheadrightarrow j_{!\ast}(\cO_U\cdot f^{-\alpha})$. Thus, by assumption there is a $w \in \Gamma(X,V^{>\alpha}\iota_+(\cO_X)_f)$ with $\mathrm{ev}_{s=-\alpha}(w) =f^{-\alpha}$. Since $\mathrm{ev}_{s=-\alpha}(f^s) = f^{-\alpha}$ and $\ker \mathrm{ev}_{s=-\alpha} = (s+\alpha)\cdot \iota_{+}(\cO_X)_f$, there exists $w' \in \Gamma(X,\iota_+(\cO_X)_f)$ such that 
\[w=f^s + (s+\alpha) w'.\]

It follows that $Z_{f} = Z_{w,f^s} - Z_{(s+\alpha)w', f^s}$. Since $w \in \Gamma(X,V^{>\alpha}\iota_+(\cO_X)_f)$, $-\alpha$ is not a pole of $Z_{w,f^s}$ by (\ref{eqn: bound of pole order}). On the other hand, applying Bernstein's argument to the conjugate term $\bar{f}^s$ \cite{bernstein} (as in \eqref{eqn: mero ext of Zmm}), we see that the order of pole at $-\alpha$ of $Z_{w', f^s}$ is at most  $\mult_{s=-\alpha} \prod_{i=0}^{r}b_f(s+i)$ with $r\gg 0$, which agrees with $\mult_{s=-\alpha}b_f(s)$ by assumption.    Thus, the order of pole of 
    \[Z_{(s+\alpha)w', f^s} = (s+\alpha)  Z_{w', f^s}\]
    at $-\alpha$ is $<\mult_{s=-\alpha} b_f(s)$, and the same holds for $Z_f$.
\end{proof}

\begin{proof}[Proof of Corollary \ref{cor: Budur Walther question for minimal exponent}]
If $\tilde{\alpha}_f\not\in \Z_{\geq 2}$, the claim follows immediately from Proposition \ref{prop: an estimate for pole order} and Theorem \ref{thm: Loesers conjecture}. Assume $\alpha:=\tilde{\alpha}_f\in \Z_{\geq 2}$, in particular $\textrm{mult}_{s=-1}b_f(s)=1$. By Theorem \ref{thm: Sabbah} and Corollary \ref{cor:Nmu}, it follows that $[f^s]\in W(\op{N})_0\gr^{1}_V\iota_{+}\cO_X$. Since the evaluation map $\mathrm{ev}_{s=-1}$ induces a surjection (see e.g. \cite[equation (23)]{LY25a})
\[ W(\op{N})_0\gr^{1}_V\iota_{+}\cO_X\twoheadrightarrow W_{\dim X+1}(\cO_X)_f/\cO_X,\]
this forces $f^{-1} \in  W_{\dim X+1}(\cO_X)_f$.  Using the $b$-function equation, we obtain $f^{-\alpha+1} \in W_{\dim X+1} (\cO_X)_f$. Thus, it is equivalent to show that $f^{-\alpha} \notin W_{\dim X+1}(\cO_X)_f$. Suppose to the contrary that this is not the case. As above, there is a surjection induced by $ \mathrm{ev}_{s=-\alpha}$:
\[W(N)_{0}\gr^{\alpha}_V\iota_{+}\cO_X\twoheadrightarrow W_{\dim X+1}\left((\cO_X)_f\cdot f^{-\alpha}\right)/\cO_X.\]
Using Proposition \ref{prop: V>alpha surjects} and Corollary \ref{cor:Nmu}, a little diagram chasing argument implies that the evaluation map $\mathrm{ev}_{s=-\alpha}$ induces a surjection (see also \cite[Theorem 4.15]{LY26}):
\[ \{ w\in V^{\alpha}\iota_+ \cO_X \mid  \textrm{mult}_{s=-\alpha}b_w(s)\leq 1\} \twoheadrightarrow W_{\dim X+1}\left((\cO_X)_f\cdot f^{-\alpha}\right)=W_{\dim X+1}(\cO_X)_f.\]
So one can find an element $w \in V^{\alpha} \iota_+ \cO_X$ so that $(s+\alpha) \cdot w \in V^{>\alpha}\iota_+ \cO_X$ with $\mathrm{ev}_{s=-\alpha}(w) =f^{-\alpha}$. Arguing as in the proof of Proposition \ref{prop: an estimate for pole order}, we obtain that the pole order of $Z_f$ at $s=-\alpha$ is $<\textrm{mult}_{s=-\alpha}\prod_{i=0}^r b_f(s+i)$ for $r\gg 0$, where the latter is $\textrm{mult}_{s=-\alpha}b_f(s)+1$.  This contradicts Theorem \ref{thm: Loesers conjecture},  which implies that $\textrm{ord}_{s=-\alpha}Z_f=\textrm{mult}_{s=-\alpha}b_f(s)+1$ under the assumption $\alpha=\tilde{\alpha}_f\in \Z_{\geq 2}$.
\end{proof}

\bibliographystyle{alpha}
\bibliography{zetafunction}

\vspace{\baselineskip}

\footnotesize{
\textsc{School of Mathematics and Statistics, University of Melbourne, Parkville, VIC, 3010, Australia} \\
\indent \textit{E-mail address:} \href{mailto:dougal.davis1@unimelb.edu.au}{dougal.davis1@unimelb.edu.au}

\vspace{\baselineskip}

\footnotesize{
\textsc{Department of Mathematics, University of Oklahoma, 660 Parrington Oval, Norman, OK 73019, United States} \\
\indent \textit{E-mail address:} \href{mailto:lorincz@ou.edu}{lorincz@ou.edu}

\vspace{\baselineskip}

\textsc{Department of Mathematics, University of Kansas, 1450 Jayhawk Blvd, Lawrence, KS 66045, United States} \\
\indent \textit{E-mail address:} \href{mailto:ruijie.yang@ku.edu}{ruijie.yang@ku.edu} 
}

\end{document}